\documentclass[a4paper]{article}

\usepackage{algorithm}
\usepackage{algpseudocode}
\usepackage{amsmath}
\usepackage{amssymb}
\usepackage{amsthm}
\usepackage{authblk}
\usepackage{bm}
\usepackage{booktabs}
\usepackage[margin=\parindent]{caption}
\usepackage{enumitem}
\usepackage{etoolbox}
\usepackage{mathtools}
\usepackage{multicol}
\usepackage{pgfplots}
\usepackage{pgfplotstable}
\usepackage{subcaption}
\usepackage{tensor}
\usepackage{tikz}
\usepackage{titlesec}
\usepackage{todonotes}
\usepackage{units}
\usepackage{wrapfig}
\usepackage{xfrac}
\usepackage{xstring}

\usepackage{hyperref}
\usepackage{cleveref}

\newcommand\Id{\text{Id}}
\makeatletter
\newcommand{\subalign}[1]{%
  \vcenter{%
    \Let@ \restore@math@cr \default@tag
    \baselineskip\fontdimen10 \scriptfont\tw@
    \advance\baselineskip\fontdimen12 \scriptfont\tw@
    \lineskip\thr@@\fontdimen8 \scriptfont\thr@@
    \lineskiplimit\lineskip
    \ialign{\hfil$\m@th\scriptstyle##$&$\m@th\scriptstyle{}##$\hfil\crcr
      #1\crcr
    }%
  }%
}

\DeclareMathOperator{\overlinecone}{\overline{cone}}
\DeclareMathOperator{\len}{Length}

\DeclareMathOperator{\Diff}{Diff}
\DeclareMathOperator{\Imm}{Imm}
\DeclareMathOperator{\dist}{dist}
\DeclareMathOperator{\distS}{dist^{\mathcal{S}}}

\theoremstyle{plain}
\newtheorem{theorem}{Theorem}
\newtheorem{lemma}[theorem]{Lemma}

\newtheorem{assumption}{Assumption}

\theoremstyle{remark}

\setlist{topsep=0pt,itemsep=0ex,partopsep=0ex,parsep=0ex}
\newlist{subassumption}{enumerate}{10}
\setlist[subassumption]{label*=(\alph*),ref=\theassumption\alph*}
\newlist{experiment}{enumerate}{10}
\setlist[experiment]{label*=(\Alph*),ref=\Alph*}

\crefname{figure}{Figure}{Figures}
\crefname{theorem}{Theorem}{Theorems}
\crefname{lemma}{Lemma}{Lemma}
\crefname{assumption}{assumption}{assumptions}
\crefname{subassumptioni}{assumption}{assumptions}
\crefname{experimenti}{test problem}{test problems}
\Crefname{experimenti}{Test problem}{Test problems}
\creflabelformat{experimenti}{(#2#1#3)}

\usetikzlibrary{decorations.pathreplacing,calligraphy,arrows,svg.path,patterns}
\usepgfplotslibrary{groupplots,colormaps,patchplots}
\usetikzlibrary{external}
\makeatletter
\renewcommand{\todo}[2][]{\tikzexternaldisable\@todo[#1]{#2}\tikzexternalenable}
\makeatother

\usepackage{xcolor}

\colorlet{light gray}{white!75!black}
\colorlet{gray}{white!50!black}
\colorlet{dark gray}{white!25!black}

\definecolor{matlab1}{RGB}{0,114,189}
\definecolor{matlab2}{RGB}{217,83,25}
\definecolor{matlab3}{RGB}{237,177,32}
\definecolor{matlab4}{RGB}{126,47,142}
\definecolor{matlab5}{RGB}{119,172,48}
\definecolor{matlab6}{RGB}{77,190,238}
\definecolor{matlab7}{RGB}{162,20,47}
\colorlet{purple}{blue!50!red!50!white!40!matlab4}
\colorlet{blue}{matlab1!90!black}
\colorlet{green}{matlab5!80!green!85!black}
\colorlet{red}{matlab2!90!black}
\colorlet{yellow}{matlab3!90!black}
\colorlet{orange}{matlab2}

\tikzset{
    ultra thin/.style = {line width=0.1pt},
    very thin/.style =  {line width=0.2pt},
    thin/.style =       {line width=0.4pt},
    semithick/.style =  {line width=0.6pt},
    thick/.style =      {line width=0.8pt},
    very thick/.style = {line width=1.2pt},
    ultra thick/.style ={line width=2.4pt},
}

\tikzstyle{tiny}   = [inner sep = 0, minimum size = 1pt]
\tikzstyle{small}  = [inner sep = 0, minimum size = 2pt]
\tikzstyle{medium} = [inner sep = 0, minimum size = 3pt]
\tikzstyle{large}  = [inner sep = 0, minimum size = 4pt]
\tikzstyle{huge}   = [inner sep = 0, minimum size = 5pt]

\tikzstyle{dashed} = [dash pattern = on 4pt off 2pt]
\tikzstyle{dash dotted} = [dash pattern = on 3pt off 1pt on 1pt off 1pt]

\pgfdeclareplotmark{circle*}{%
	\pgfpathcircle{\pgfpoint{0pt}{0pt}}{\pgfplotmarksize}%
	\pgfusepathqfillstroke%
}
\pgfdeclareplotmark{triangle*}{%
	\pgfpathmoveto{\pgfpointpolar{-90}{\pgfplotmarksize}}%
	\pgfpathlineto{\pgfpointpolar{-210}{\pgfplotmarksize}}%
	\pgfpathlineto{\pgfpointpolar{30}{\pgfplotmarksize}}%
	\pgfpathclose
	\pgfusepathqfillstroke%
}
\pgfdeclareplotmark{ast}{%
	\foreach \i in {1,2,3,4,5} {%
		\pgfpathmoveto{\pgfpointorigin}%
		\pgfpathlineto{\pgfpointpolar{90+2*72*\i}{\pgfplotmarksize}}%
	}%
	\pgfusepathqstroke%
}

\pgfplotsset{
	line 1/.style = {
		mark=circle*, 
		mark size=1.5pt,
		blue,
		mark options={fill=white},
	},
	line 2/.style ={
		mark=triangle*,
		mark size=2.0pt,
		green,
		mark options={fill=white},
	},
	line 3/.style ={
		mark=circle*, 
		mark size=1.5pt,
		yellow,
		mark options={fill=yellow!70!white},
	},
	line 4/.style = {
		mark=triangle*,
		mark size=2.0pt,
		red,
		mark options={fill=red!70!white},
	},
	line 5/.style = {
		mark = ast,
		mark size=2.5pt, 
		black,
	},
	smooth line 1/.style = {smooth,blue},
	smooth line 2/.style = {smooth,green},
	smooth line 3/.style = {smooth,yellow},
	smooth line 4/.style = {smooth,red},
}
\pgfplotscreateplotcyclelist{high contrast marks}{%
	{line 1},
	{line 2},
	{line 3},
	{line 4},
	{line 5},
}
\pgfplotscreateplotcyclelist{local lines}{%
	smooth line 1,
	smooth line 2,
	smooth line 3,
	smooth line 4
}

\pgfplotsset{compat=newest}

\pgfplotsset{
	convergence plot/.style = {
		xmode=log,
		ymode=log,
		at = {(0,0)},
		xmin = 1e1,
		xmax = 1e4,
		xmajorgrids,
		ymajorgrids,
		xminorgrids=false,
		yminorgrids=false,
		grid style = {thin,light gray},
		axis lines = left,
		axis line style = {draw=none},
		minor tick style = {thin,light gray},
		minor tick length = 1.5pt,
		major tick length = 3pt,
		major tick style = {draw=none},
		tick align = inside,
		max space between ticks=20,
		xlabel = {$h^{-1}$},
		legend pos = outer north east,
		legend style = {
			draw=none,
		},
		legend cell align={left},
		width = 2.7cm,
		height = 3.6cm,
	},
	geodesic/.style = {
		mesh,
		patch,
		patch type=cubic spline,
		line cap=round,
		point meta=explicit,
	},
	title style={at={(0.5,0)},anchor=north,yshift=-1.2cm},
	every axis/.append style={
		execute at begin axis={
			\pgfplotsset{
				legend style={
					font=\footnotesize
				},
				every x tick label/.append style={font=\footnotesize},
				every y tick label/.append style={font=\footnotesize},
			}
		}
	},
	scale only axis,
	every axis plot/.append style={thick},
	cycle list name=high contrast marks,
	anchor=center,
	colormap={hues}{
        color(0cm)=(matlab7)
        color(2cm)=(matlab2)
        color(4cm)=(matlab3)
        color(7cm)=(matlab5)
        color(10cm)=(matlab1)
        color(12cm)=(matlab4)
        color(14cm)=(matlab7)
    }
}

\title{A PDE-based Method for Shape Registration}
\author{
    Esten Nicolai Wøien\thanks{Department of Mathematical Sciences, Norwegian University of Science and Technology, Trondheim, Norway} \thanks{\url{esten.n.woien@ntnu.no}}
    \hspace{1cm}
    Markus Grasmair$^*$\thanks{\url{markus.grasmair@ntnu.no}}}
\date{March, 2021}

\begin{document}

\maketitle


\begin{abstract}
	In the square root velocity framework, the computation of shape space distances and the registration of curves requires solution of a non-convex variational problem. In this paper, we present a new PDE-based method for solving this problem numerically. The method is constructed from numerical approximation of the Hamilton-Jacobi-Bellman equation for the variational problem, and has quadratic complexity and global convergence for the distance estimate. In conjunction, we propose a backtracking scheme for approximating solutions of the registration problem, which additionally can be used to compute shape space geodesics. The methods have linear numerical convergence, and improved efficiency compared previous global solvers.
\end{abstract}

\section{Introduction}

A large number of applications require the manipulation
and mathematical or statistical analysis of geometric objects
in general and curves in particular.
Examples from mathematical image processing are
segmentation, where one wants to find and classify different
objects within an image based on their outlines (see e.g.~\cite{Kass88}
for a classical model), or object tracking (see \cite{tracking}), where one wants to follow the same object over a sequence of consecutive frames. Other examples include the analysis of shapes of proteins \cite{Liu11}, modelling and analysis of computer animations \cite{Bauer17,Celledoni16}, or also inverse problems concerning the detection of shapes from indirect measurements \cite{Eckhardt19}.

In order to perform these tasks, it is necessary to have
a well-defined and easily computable notion of distance between
curves at hand.
One important example is the \emph{Square Root Velocity} (SRV) distance
originally introduced in \cite{srvt2010,Mio2007}
(see Section~\ref{se:prelim} below for a precise definition),
which can be interpreted as a measure for the
bending and stretching energy that is required for transforming
one curve into another.
For parametrised curves, this distance is defined by applying first
a non-linear transformation---the \emph{Square Root Velocity transform}
(SRVT)---to the involved curves, which maps
them onto the unit sphere in $L^2$. Then, the distance of the
curves is defined as the unit sphere distance of their SRV transformations.
Even more, this setting makes it possible to regard the space
of all parametrised curves as a manifold with a Riemannian structure that
is inherited from the unit sphere in $L^2$.
In particular, one can define geodesics between parametrised
curves, that is, optimal deformations of one curve into another.

However, in many applications we are only interested in the image of
a curve, but not the concrete parametrisation.
We thus rather require a distance between \emph{shapes},
that is, equivalence classes of curves modulo reparametrisations.
Within the SRV framework, this can be achieved by defining the
distance between two shapes as the infimum
of the distance between all curves within their equivalence class.
Or, given two parametrised curves $c_1$ and $c_2$, we define the distance
between their shapes as the infimum of the distance between
all reparametrisations of $c_1$ and $c_2$.
It can be shown that this
infimum is positive for all distinct shapes
and thus defines a distance on the set of all shapes.
Moreover, it is again possible to view the space of all
shapes as a Riemannian manifold and thus to define geodesic
between shapes.
We refer to~\cite{Srivastava16} for a detailed introduction
into shape analysis within the SRV framework; a short overview
can also be found in \cite{bauer2020intrinsic}.

The actual computation of the shape distance
and of geodesics, though, requires the solution of an
optimisation problem over the space of all reparametrisations.
This problem has the form
\begin{equation}\label{eq:intromin}
  \inf_{\varphi_1,\,\varphi_2} \int_I F(\varphi_1(t),\varphi_2(t),\varphi_1'(t),\varphi_2'(t))\,dt,
\end{equation}
where the infimum is computed over all orientation preserving
reparametrisations of the unit interval $I = [0,1]$.
Here the integrand $F$ depends on the SRVTs of the curves
$c_1$ and $c_2$ one wants to compare.
Due to invariance properties of the SRV distance,
it is sufficient to compute the minimum in~\eqref{eq:intromin}
only with respect to one of
the diffeomorphisms, e.g.~w.r.t.~$\varphi_1$ while
leaving $\varphi_2$ constant equal to $\Id$.
This reduces the dimensionality of the problem,
but one is still left with an optimisation problem
over a function space.

For the numerical solution, there are two main approaches:
gradient based methods and dynamical programming.
In the dynamic programming approach
introduced in~\cite{Mio2007} (see also~\cite{younes1998} for a similar numerical
approach for a different shape distance), one approximates the diffeomorphism
$\varphi_1$ by a piecewise linear approximation with nodal points and nodal values
within a fixed partition $0 = t_0 < t_1 < \ldots < t_N = 1$ of the unit
interval. The resulting discrete optimisation problem is then
solved by a dynamic programming algorithm.
Without further modifications, this algorithm has a
time complexity of $O(N^4)$ and thus is not useful for practical
applications.
A significant speed-up is possible, though, by limiting
the set of possible slopes for the linear approximations of $\varphi_1$.
In fact, a method with complexity $O(N^3)$ has already been proposed in~\cite{Mio2007}.
Even more, a variant with complexity $O(N)$ has been presented in~\cite{dogan2016}
(see also \cite{dogan2015}), which is an iterative based on an adaptive, local refinement
of the search grid for the dynamic program.

Gradient based methods usually work on a finite dimensional
approximation of the space of all reparametrisations,
e.g.~using B-splines or trigonometrical functions.
The optimisation problem~\eqref{eq:intromin} is then
rephrased as a problem for the basis coefficients.
This is a finite dimensional optimisation problem,
which can, in principle, be solved with standard methods like
gradient descent or quasi-Newton methods.
One difficulty is the constraint that the functions
$\varphi$ we are optimising over are orientation preserving
diffeomorphisms and thus monotonically increasing.
Thus one has a positivity constraint for $\varphi'$,
which is difficult to handle numerically.
Thus~\cite{Huang16} rephrase the problem in terms of
$\gamma^2 := \varphi'$, which yields an optimisation problem
with the single equality constraint that $\lVert\gamma\rVert_{L^2}^2 = 1$.

We note here, though, that the optimisation problem~\eqref{eq:intromin} is
typically highly non-convex, as the reparametrisations appear as arguments
of the curves or their SRVTs. Thus it is highly likely
that there is a large number of local minimisers and
other critical point.
Indeed, an example of such a situation is shown in
our numerical examples in Section~\ref{se:localsols}.
Since gradient based methods are local,
it is therefore necessary to initialise the iteration
with a sufficiently good initial guess of the solution.
The same holds for the adaptively refined dynamic program
suggested in~\cite{dogan2016}.
If the initialisation of that method is too coarse,
it can happen that the refinement strategy is never
able to find the true global minimum.
\medskip

In this paper, we want to present, and analyse, an
alternative approach to the solution of~\eqref{eq:intromin}
which is based on a formulation as a continuous dynamic program.
This formulation allows us to define a continuous \emph{value function}
$u \colon I \times I \to \mathbb{R}$, where $u(x_1,x_2)$
measures the minimal (partial) cost of a reparametrisation
satisfying $\varphi(x_1) = x_2$.
In particular, $u(1,1)$ is precisely the value of the optimisation
problem~\eqref{eq:intromin}.
A precise definition of $u$ is given in~\eqref{eq:value} below.
This value function $u$ has been shown in~\cite{calder2013} to satisfy, in the
viscosity sense, the associated Hamilton--Jacobi--Bellman equation,
which is a hyperbolic
PDE with boundary values given for $x_1 = 0$ and $x_2 = 0$.
Moreover, convergent numerical schemes for the solution of that PDE have
been proposed in~\cite{calder2017,calder2013,thawinrak2017}.

The main contribution of this paper is a generalisation
of these schemes, which is closer in spirit to the definition
of $u$ by means of a dynamical program. In particular, our
approach allows it to recover the optimal reparametrisation
in a natural way by solving an ODE.
Since the value function is defined by means of the dynamic
programming principle, the resulting numerical method
will be globally convergent.
At the same time, the formulation as a PDE allows the
numerical schemes to have a time complexity of only $O(N^2)$.
In Section~\ref{se:prelim}, we will formulate the necessary
mathematical background in shape analysis and the SRVT.
The value function and the Hamilton--Jacobi--Bellman equation
are introduced in Section~\ref{se:dynprog}. In Sections~\ref{sec:schemes}
and~\ref{se:proposed}, we will first introduce the general numerical
framework together with a convergence analysis, and then propose
concrete numerical schemes. Then we will discuss
the recovery of the optimal reparametrisation from the value function
and the construction of geodesics in shape space in Section~\ref{se:backtracking}.
Finally, we will present numerical experiments in Section~\ref{se:experiments}.

\section{Preliminaries}
\label{se:prelim}

In the following, we will provide a brief introduction into shape analysis
using the square-root-velocity-transform.
To that end, let $I = [0,1]$ be the unit interval and $d \in \mathbb{N}$, and denote by
\[
\Imm(I;\mathbb{R}^d) := \bigl\{ c \in C^1(I;\mathbb{R}^d) : \lvert \dot{c}(t) \rvert > 0 \text{ for all } t \in I\bigr\}
\]
the space of all $C^1$ immersions of $I$ in $\mathbb{R}^d$.
We define the (scaled) \emph{Square-Root-Velocity-Transform} (SRVT)
\[
Q \colon \Imm(I;\mathbb{R}^d) \to C(I;\mathbb{R}^d\setminus\{0\})
\]
as
\begin{equation}\label{eq:SRVT}
q(t) = Q(c)(t) := \frac{1}{\sqrt{\len(c)}}\frac{\dot{c}(t)}{\sqrt{\lvert \dot{c}(t)\rvert}},
\end{equation}
where
\[
\len(c) := \int_I \lvert \dot{c}(t)\rvert\,dt
\]
denotes the length of the curve $c$.
Noting that
\[
\int_I \lvert q_i(t)\rvert^2\,dt = \int_I \frac{\lvert \dot{c}_i(t)\rvert}{\len(c_i)}\,dt = 1,
\]
we see that, actually, the SRVT maps a curve to an element
of the unit sphere in $L^2(I;\mathbb{R}^d)$.
The (scaled) \emph{Square-Root-Velocity} (SRV) distance between $c_1$ and $c_2$ is now defined as the geodesic distance between $q_1$ and $q_2$, that is,
\begin{equation}\label{eq:dist2}
\dist(c_1,c_2) := \arccos\Bigl(\int_I \langle q_1(t),q_2(t)\rangle\,dt\Bigr).
\end{equation}

Obviously, the SRV distance is translation invariant. Moreover,
due to the scaling by the square root of the length of the curves,
it is easy to see that it is scale invariant, that is,
\[
\dist(\lambda_1 c_1,\lambda_2 c_2) = \dist(c_1,c_2)
\]
for all $\lambda_1$, $\lambda_2 > 0$.
Since $\dist$ is defined by means of a geodesic distance,
it also satisfies the triangle inequality.
As a consequence, it follows that $\dist$ is a metric on the \emph{pre-shape space}
$\Imm(I;\mathbb{R}^d) / G$, where $G$ denotes the group of translations
and scalings in $\mathbb{R}^d$.

In fact, one can show that it is possible to regard
$\Imm(I;\mathbb{R}^d)/G$ as a Riemannian manifold
for which $\dist$ is the geodesic distance.
This also makes it possible to construct geodesics between
curves:
Consider $c_1$, $c_2 \in \Imm(I;\mathbb{R}^d) / G$ 
with SRVTs $q_1 = Q(c_1)$ and $q_2 = Q(c_2)$.
If $q_1(t) \neq -q_2(t)$ for all \(t\in I\)
then the geodesic between $c_1$ and $c_2$ is given as
\begin{equation}\label{eq:geodesic}
\tau \mapsto c_\tau := Q^{-1}\bigl(w(1-\tau)q_1 + w(\tau) q_2\bigr) \in \Imm(I;\mathbb{R}^d)/G.
\end{equation}
Here $Q^{-1}$ is the inverse SRVT, which can be explicitly computed as
\[
Q^{-1}(q)(t) = \int_0^t \lvert q(t')\rvert q(t')\,dt',
\]
and
\[
  w(\tau) = \frac{\sin\bigl(\tau \dist(c_1,c_2)\bigr)}{\sin\bigl(\dist(c_1,c_2)\bigr)}.
\]
We note that~\eqref{eq:geodesic} still makes sense if $q_1(t) = -q_2(t)$ for some (though not all) $t \in I$. In that case, however, the resulting curves $c_\tau$ will not all be immersions.
\medskip

Next we define the \emph{shape space}
\[
\mathcal{S}(I;\mathbb{R}^d) := \Imm(I;\mathbb{R}^d) / (\Diff_+(I) \times G),
\]
where
\[
\Diff_+(I) = \bigl\{\varphi \in C^\infty(I) : \varphi(0) = 0,\ \varphi(1) = 1,\, \varphi'(t) > 0 \text{ for all } t\in I\bigr\}
\]
is the group of orientation preserving diffeomorphisms of $I$.
Given two shapes $[c_1]$, $[c_2] \in \mathcal{S}(I;\mathbb{R}^d)$, we
then define their distance as
\begin{equation}\label{eq:distS1}
\distS([c_1],[c_2]) := \inf_{\varphi_1,\,\varphi_2 \in\Diff_+(I)} \dist(c_1\circ \varphi_1,c_2\circ \varphi_2).
\end{equation}

In fact, it is possible to simplify this expression,
as the SRV distance is invariant under simultaneous reparametrisations 
in the sense that
\[
\dist(c_1,c_2) = \dist(c_1\circ\varphi,c_2\circ\varphi)
\qquad\text{ for all } \varphi \in \Diff_+(I).
\]
Thus we have that
\begin{equation}\label{eq:distS}
\distS([c_1],[c_2]) = \inf_{\varphi \in \Diff_+(I)} \dist(c_1\circ \varphi,c_2).
\end{equation}
Again, one can show that this distance is induced by a Riemannian
metric on the shape space $\mathcal{S}(I;\mathbb{R}^d)$,
which in turn makes it possible to define geodesics between
certain shapes.
If the infimum in~\eqref{eq:distS} is attained at a diffeomorphism
$\varphi_{\text{opt}}$, then the geodesic between the shapes $[c_1]$
and $[c_2]$ is the equivalence class of the geodesic
between $c_1 \circ \varphi_{\text{opt}}$ and $c_2$.
Explicitly, this is given as
\[
\tau \mapsto \bigl[Q^{-1}\bigl(w^{\mathcal{S}}(1-\tau) (q_1 \circ \varphi_{\text{opt}})\sqrt{\smash{\varphi'_{\text{opt}}}\vphantom)}  + w^{\mathcal{S}}(\tau) q_2\bigr)\bigr]
\]
with
\[
  w^{\mathcal{S}}(\tau) = \frac{\sin\bigl(\tau \distS([c_1],[c_2])\bigr)}{\sin\bigl(\distS([c_1],[c_2])\bigr)}.
\]

In general, though, the infimum in~\eqref{eq:distS}
is not attained in $\Diff_+(I)$.
However, it was shown in \cite{bruveris2016}
that a relaxation of the optimisation problem~\eqref{eq:distS1} 
to a larger space of reparametrisations attains its minimum.
Denote to that end
\[
\Phi := \bigl\{\varphi \in AC(I) : \varphi(0) = 0,\ \varphi(1) = 1,\
\varphi'(t) \ge 0 \text{ for a.e. } t \in I\bigr\}.
\]

\begin{lemma}[Bruveris 2016]
  Assume that $c_1$, $c_2 \in \Imm(I;\mathbb{R}^d)$.
  Then 
  \begin{equation}\label{eq:minprob}
  \distS([c_1],[c_2]) = \inf_{\varphi_1,\, \varphi_2 \in \Phi} \dist(c_1 \circ \varphi_1,c_2 \circ \varphi_2).
  \end{equation}
  Moreover, the optimisation problem in~\eqref{eq:minprob} admits
  a solution $(\varphi_1,\varphi_2) \in \Phi^2$.
\end{lemma}

The main topic of this paper is the efficient numerical solution
of the optimisation problem~\eqref{eq:minprob}.
In order to do so, we use a more explicit formulation of~\eqref{eq:minprob}.
By applying the chain rule in~\eqref{eq:SRVT}, one obtains that
\[
Q(c\circ \varphi)(t) = Q(c)(\varphi(t))\sqrt{\varphi'(t)}.
\]
Thus, the optimisation problem~\eqref{eq:minprob} reads explicitly as
\[
  \inf_{\varphi_1,\,\varphi_2 \in \Phi} \arccos\Bigl(\int_I
  \Bigl\langle q_1\bigl(\varphi_1(t)\bigl)\sqrt{\varphi_1'(t)},
  q_2\bigl(\varphi_2(t)\bigr)\sqrt{\varphi_2'(t)}\Bigr\rangle\,dt\Bigr).
\]
Since $\arccos$ is monotonically increasing,
we can alternatively define
\begin{equation}\label{eq:Jdef}
J(\varphi_1,\varphi_2) := \int_I \bigl\langle q_1\bigl(\varphi_1(t)\bigr),q_2\bigl(\varphi_2(t)\bigr)\bigr\rangle
\sqrt{\varphi_1'(t)\varphi_2'(t)}\,dt
\end{equation}
and compute
\begin{equation}\label{eq:Jmax}
\distS([c_1],[c_2]) = \arccos\Bigl(\sup_{\varphi_1,\,\varphi_2\in\Phi}\!\!\!\! J(\varphi_1,\varphi_2)\Bigr).
\end{equation}
Moreover, it is shown in \cite[Proof of Prop.~15]{bruveris2016} that
one can replace the functional $J$ in~\eqref{eq:Jmax} with its convex relaxation
\[
J_c(\varphi_1,\varphi_2) :=
\int_I \max\bigl\{\bigl\langle q_1\bigl(\varphi_1(t)\bigr),q_2\bigl(\varphi_2(t)\bigr)\bigr\rangle,0\bigr\}
\sqrt{\varphi_1'(t)\varphi_2'(t)}\,dt
\]
without affecting the minimisers.

\section{A General Variational Problem}
The problem \eqref{eq:Jmax} can be seen as a special case of the variational problem
\begin{equation}
	\sup_{\bm\varphi\in\mathcal A} \Biggl(J(\bm\varphi) :=
		\int_I f(\varphi_1(t),\varphi_2(t))
		\sqrt{\varphi_1'(t)\varphi_2'(t)}\,dt\biggr).
	\label{eq:prob}
\end{equation} 
Here, we denote \(\bm\varphi = (\varphi_1,\varphi_2)\) and \(\mathcal A = \Phi\times\Phi\). Moreover, $f\colon I \times I \to \mathbb{R}_{\ge 0}$ is a continuous, non-negative function.
Although with a different motivation, this problem has been studied in \cite{calder2013,calder2017,thawinrak2017,deuschel1995}. In \cite{calder2013}, a Hamilton-Jacobi-Bellman (HJB) formulation of this problem was derived, and HJB-based solvers were constructed in \cite{calder2013,calder2017,thawinrak2017}. In the following, we recall the main HJB-related results from \cite{calder2013}, and provide some useful generalisations of the results.

\subsection{Dynamic Programming and the Value Function}\label{se:dynprog}

For variational problems of the type \eqref{eq:prob}, the solution can be described using dynamic programming. The starting point is the introduction of a value function \(u:[0,1]^2\to\mathbb R\) defined as
\begin{equation}
	\label{eq:value}
	u(t,\bm x) :=
	\sup_{\bm \varphi\in\mathcal A(t,\bm x)}  \int_0^t f (\varphi_1,\varphi_2) \sqrt{\smash{\varphi_1'\varphi_2'}\vphantom f}\,dt
\end{equation}
with 
\[
	\mathcal A(t,\bm x) := \bigl\{\bm\varphi \in AC(I;\,\mathbb R^2) : \bm\varphi(0) = \bm0,\ \bm\varphi(t) = \bm x,\
\varphi'(s) \ge 0 \text{ for a.e. } s \in I\bigr\}.
\]
Due to the reparametrisation invariance of the integral, we have that \(u(t,\bm x)\) is independent of \(t\). We will therefore omit the time variable in the definition of \(u\) and \(\mathcal A\) and simply write \(u(\bm x) = u(1,\bm x)\). 

In the case when either \(x_1=0\) or \(x_2=0\), we have that \(\varphi_1'\varphi_2'=0\) a.e. for all admissible paths. Consequently, the integrand is zero almost everywhere, meaning that we obtain the boundary values \(u(0,x_2) = u(x_1,0) = 0\). Furthermore, the value function satisfies the dynamic programming principle,
\begin{equation}
		\label{eq:dpp}
		u(\bm\varphi(t)) \geq u(\bm\varphi(t-h)) + \int_{t-h}^t f(\varphi_1,\varphi_2)\sqrt{\varphi_1'\varphi_2'}\,dt,
\end{equation}
for all \(\bm\varphi\in \mathcal A\). Moreover, \(\bm\varphi\) is a solution of \eqref{eq:prob} if and only if we have equality for all \(t\) and \(h\). Dividing by \(h\), and taking the limit as \(h\to0\), this means that a solution \(\bm\varphi\) formally satisfies the differential equation
\[
	-\frac{d}{dt}u(\bm\varphi) + f(\bm\varphi)\sqrt{\varphi_1'\varphi_2'} = 0,
\]
which for smooth \(u\) reads
\[
	-Du(\bm\varphi)\cdot\bm\varphi' + f(\bm\varphi)\sqrt{\varphi_1'\varphi_2'} = 0.
\]
This means that it is possible to reconstruct \(\bm\varphi\) from the value function. 

We will now discuss some properties of the value function that will be needed later in the paper. First of all, the dynamic programming principle \eqref{eq:dpp} implies immediately that $u$ is monotone non-decreasing in the sense that \(u(\bm x) \geq u(\bm y)\) whenever \(\bm x\geq \bm y\). Additionally, wherever \(f(\bm x) >0\), the value function is locally strictly increasing: if \(x_i>y_i\) element-wise then \(u(\bm x)>u(\bm y)\). Finally it has been shown that \(u(\bm x)\) is H\"older continuous with exponent \(\frac12\) \cite[Lemma 1]{calder2013} while \(v(\bm x) := u(\bm x)^2\) is Lipschitz continuous \cite[Lemma 9]{calder2017}.

\subsection{The HJB equation}
For variational problems such as \eqref{eq:prob}, the value function can often be described as the unique solution of the associated Hamilton-Jacobi-Bellman equation. For a general problem of the form
\[
	\sup_{\bm\varphi}\int_0^1 \ell(\bm\varphi(t),\bm\varphi'(t))dt
\]
with associated time dependent value function \(u(t, \bm x)\), this reads
\[
	-u_t(t,\bm x) + \sup_{\bm\alpha\in A} \bigl(- Du(\bm x)\cdot\bm\alpha + \ell(\bm x,\bm\alpha)\bigr) = 0.
\]
Here, \(\bm x\) and \(\bm\alpha\) correspond to \(\bm\varphi(t)\) and \(\bm\varphi'(t)\), respectively. 

As discussed above, the value function is in our case time independent. Thus, we would expect a stationary Hamilton-Jacobi-Bellman equation of the form
\[
	\begin{cases}
	\begin{alignedat}{2}
		H(\bm x,Du) 
		&= 0, \quad && \text{in } (0,1]^2,\\
		u(0,x_2) = u(x_1,0) &= 0,
	\end{alignedat}
	\end{cases}
\]
with the Hamiltonian
\[
	H(\bm x,\bm p) = \sup_{\bm \alpha\in \mathbb R_{\geq0}^2} -\bm p\cdot\bm \alpha+ f(\bm x)\sqrt{\alpha_1\alpha_2}.
\]
The restriction \(\bm \alpha\in \mathbb R_{\geq0}^2\) follows from the fact that \(\bm\varphi'(t)\) takes values in \(\mathbb R_{\geq0}^2\).
However, since the functional is positively homogeneous in \(\bm\alpha\), this leads to a degenerate Hamiltonian which only takes values \(H(\bm x,\bm p) \in \{0,+\infty\}\). This property is a consequence of the reparametrisation invariance of the problem, and will be a problem for uniqueness of viscosity solutions of the HJB equation. On the other hand, due to the reparametrisation invariance, we are able to impose restrictions to the admissible space. Therefore, for some well chosen set \(A\) representing the admissible derivatives of the paths \(\bm\varphi\), we define the Hamiltonian as
\[
	H(\bm x,\bm p) = \sup_{\bm \alpha\in A} -\bm p\cdot \bm\alpha + f(\bm x)\sqrt{\alpha_1\alpha_2}.
\]
We require \(A\) to have certain properties. 
\begin{itemize}
	\item First of all, \(A\) should reflect the admissible directions of the path \(\bm\varphi\). In particular, we must allow for all monotone increasing directions, which means that we must have that \(\overlinecone A = \mathbb R_{\geq0}^2\). 

	\item Secondly, we want the admissible set to permit both negative and positive values for the Hamiltonian (to avoid redundancy of viscosity sub- and supersolutions). This requires that the set $A$ is bounded away from the origin.

	\item Lastly, we want the admissible set to be compact to allow for the Hamiltonian to have a maximiser \(\bm\alpha\). This is not needed for the viscosity characterisation of the value function, but will be a necessary assumption in the construction of numerical solvers. 
\end{itemize}
Together, this can be summarised in the following two assumptions:
\begin{assumption}
	\label{asm:a}
	 \(A\) satisfies the following:
\begin{subassumption}
\item
	\label{asm:ahjb}
	\(A\subset \mathbb R_{\geq0}^{2}\) such that \(\overlinecone A = \mathbb R_{\geq0}^{2}\) and \(\inf_A |\bm\alpha| > 0\).
	\item
	\label{asm:acompact}
	\(A\) compact.
\end{subassumption}
\end{assumption}

There are a few examples of admissible sets which satisfy these assumptions. The natural choices are
\[
	A_r = \{\bm\alpha\in\mathbb R_{\geq0}^2\mid |\bm\alpha|_r=\text{const.}\},
\]
with \(r\in\{1,2,+\infty\}\). More general, for \(1\leq r\leq +\infty\), \(A_r\) satisfies both \cref{asm:ahjb,asm:acompact}. Additionally, there are options satisfying only \cref{asm:ahjb} including
\begin{align*}
	A &= \{\bm\alpha\in\mathbb R_{\geq0}^2\mid \alpha_1\alpha_2 = \text{const.}\}, \\
	 A &= \{\bm\alpha\in\mathbb R_{\geq0}^2\mid \alpha_1=\text{const.}\}. 
\end{align*}
The first choice is (implicitly) used in \cite{calder2013}, while the second choice corresponds to the restriction \(\varphi_1'=1\), which is common in the literature of shape analysis.
However, both these cases lead to special situations where the maximum of \(H\) is not necessarily attained by any \(\bm\alpha\). 

We have the following result:
\begin{theorem}
	\label{thm:hjb}
	Assume that \(A\) is such that \cref{asm:ahjb} holds. Then, the value function \(u\) is the unique viscosity solution of the hyperbolic PDE
	\begin{equation}
		\label{eq:hjb}
		\begin{cases}
		\begin{alignedat}{2}
			H(\bm x,Du) 
			&= 0, \quad && \text{in } (0,1]^2,\\
			u(0,x_2) = u(x_1,0) &= 0.
		\end{alignedat}
		\end{cases}
	\end{equation}
\end{theorem}

That  \(u\) is a viscosity solution of the PDE is proved in \cite[Theorem 2]{calder2013} while uniqueness follows from \cite[Theorem 3]{calder2013}. The actual results in \cite{calder2013} are formulated using a Hamiltonian obtained from the choice \(A = \{\bm\alpha\in\mathbb R_{\geq0}^2\mid \alpha_1\alpha_2 = \text{const.}\}\) after some equivalent reformulations.\footnote{This choice of admissible set gives \(H(\bm x,\bm p) = -\sqrt{\max\{p_1,0\}\max\{p_2,0\}} + \frac12f(\bm x)\) up to some constant, while \cite{calder2013} uses \(H(\bm x,\bm p) = -\max\{p_1,0\}\max\{p_2,0\} + \frac14f(\bm x)^2\). These Hamiltonians will always have the same sign, meaning that they are equivalent in the viscosity sense. } However, as used in the proof, the positive homogeneity of the functional of \(H\) with respect to \(\bm\alpha\) implies equivalence of viscosity solutions for all choices of \(A\) satisfying \cref{asm:a}.


\section{Monotone Schemes for the HJB Equation}
\label{sec:schemes}
We will now construct a new family of schemes for solving the HJB equation. 
The schemes have can be interpreted both as finite difference approximations to the HJB equation similar to the schemes of \cite{calder2017, calder2013, thawinrak2017}, but also as approximations to the dynamic programming principle \eqref{eq:dpp}.

\subsection{Schemes based on \texorpdfstring{\boldmath\(Du\)}{Du}}
\label{sec:schemeu}
We start by constructing numerical schemes for approximating solutions to the HJB equation. 
As in \cite{calder2017,calder2013,thawinrak2017}, we discretise the unit square into a square grid \([0,1]_h^2 := \{0,h,2h,\ldots,1\}^2\). Here, we assume that \(h=1/N\), \(N\) being the number of discretisation points. For each grid node \(\bm x\), we solve a finite difference approximation to the HJB equation, which takes the form
\begin{equation}
	\label{eq:scheme0}
	\max_{\bm\alpha\in A} 
		-D^- u(\bm x)\bm\alpha + 
		f(\bm x)\sqrt{\alpha_1\alpha_2} = 0.
\end{equation}
Here we use the backward difference approximation 
\begin{equation}
	\label{eq:approx_derivative}
	-D^-u(\bm x)\bm\alpha := \frac{u(\bm x-h\bm\alpha) - u(\bm x)}{h}.
\end{equation}
For smooth \(u\), this is a first order approximation to \(-Du(\bm x)\bm\alpha\). However, the term \(u(\bm x-h\bm\alpha)\) needs to be approximated as \(\bm x-h\bm\alpha\) will not coincide with a grid point for all values of \(\bm\alpha\). We will denote this approximation as \(g_h(\bm x,\bm\alpha,u)\). Inserting the approximation \eqref{eq:approx_derivative} into \eqref{eq:scheme0}, this then gives the general scheme 
\begin{equation}
	\label{eq:scheme}
	\max_{\bm\alpha\in A} 
		\frac{g_h(\bm x,\bm\alpha,u) - u(\bm x)}{h} + 
		f(\bm x)\sqrt{\alpha_1\alpha_2} = 0.
\end{equation}
After rearranging the terms, this results in the expression
\begin{equation}
	\label{eq:schemesol}
	u(\bm x) = \max_{\bm\alpha\in A} g_h(\bm x,\bm\alpha,u) + hf(\bm x)\sqrt{\alpha_1\alpha_2}.
\end{equation}
The above idea is in contrast to the typical approach as in \cite{calder2017,calder2013,thawinrak2017}, where \(D u(\bm x)\), interpreted as a gradient, is approximated numerically. The directional derivative is then computed as the inner product of \(\bm\alpha\) with the approximation to \(Du(\bm x)\). In fact, the approximation \eqref{eq:approx_derivative} can be seen as a generalisation of the typical approach, as we can always approximate \(u(\bm x-h\bm\alpha)\) using
\[
	g_h(\bm x,\bm\alpha,u) = u(\bm x) - hD^-u(\bm x)\cdot \bm\alpha .
\]
for some approximate gradient \(D^-u(\bm x)\).

In order to prove convergence of the schemes, we will use the classical proof of Barles \& Souganidis  for so-called monotone schemes (see \cite{barles1991}) To start, we denote the schemes as \(S_h(\bm x,u_h(\bm x),u_h) = 0\) with 
\begin{align*}
	S_h(\bm x,t,u) =
	 \max_{\bm\alpha\in A} 
		\frac{g_h(\bm x,\bm\alpha,u)- t}{h} + 
		f(\bm x)\sqrt{\alpha_1\alpha_2}.
\end{align*}
\cite[Theorem 2.1]{barles1991} states that if a scheme is \emph{monotone}, \emph{stable} and \emph{consistent} it is also convergent. Here, we define
\begin{itemize}
	\item \emph{Monotonicity}: \(S_h\) is non-decreasing in \(u\).
	\item \emph{Stability}: The scheme \(S_h\big(\bm x,u_h(\bm x),u_h\big) = 0\) has a solution \(u_h\) for which \(\|u_h\|_\infty \leq \text{const.}\) independent of \(h\). 
	\item \emph{Consistency}: For every \(\psi \in C^\infty\), we have that
	\[
		\lim_{h\to0, \bm y\to\bm x,\xi\to0} 
			S_h\big(\bm y,\psi(\bm y)+\xi,\psi+\xi\big) = H(\bm x, D\psi(\bm x)).
	\] 	 	
\end{itemize}

Given \(A\), the scheme \eqref{eq:schemesol} is completely determined by \(g_h\). Accordingly, monotonicity, stability and consistency of the scheme can be inferred from the properties of \(g_h\). 

\begin{assumption}
	\label{asm:g}
	The approximation \(g_h : [0,1]_h^2\times A \times C_{\frac12}[0,1]_h^2 \to \mathbb R\) satisfies:
\begin{subassumption}
		\item \(g_h\) is monotone non-decreasing in \(u\).
		\label{asm:g:monotone}
		\item \(g_h\) is localised: there exists \(C>0\) such that, for all functions \(\psi,\xi \colon [0,1]^2\to \mathbb R\) that satisfy \(\psi=\xi\) on the ball \(B_{Ch}(\bm x)\) of radius $Ch$ centered at $\bm x$, we have that
		\[
			g_h(\bm x, \bm\alpha, \psi) = g_h(\bm x, \bm\alpha, \xi).
		\]
		\label{asm:g:localised}
		\item For constant \(\psi = \psi_0\), we have that
		\(
			g_h(\bm x,\bm\alpha,\psi) = \psi_0.
		\)
		\label{asm:g:constant}
		\item \(g_h\) is a superlinear approximation: for every \(\psi \in C^\infty([0,1]^2)\) and all \(L>0\), there exists a modulus of continuity \(\omega_{\psi,L}\) such that
		\[
		\biggl| 
			\frac{g_h\big(\bm x,\bm\alpha,\psi+\xi\big) - \psi(\bm x-h\bm\alpha)-\xi}{h} 
		\biggr|
		 \leq \omega_{\psi,L}(h)
	\]
	for every \(|\xi|\leq L\). 
		\label{asm:g:superlinear}
	\end{subassumption}
\end{assumption}

\begin{theorem}
	\label{thm:scheme}
	Under \cref{asm:a,asm:g}, the scheme \eqref{eq:schemesol} is convergent. 
\end{theorem}
\begin{proof}
	The scheme satisfies the following properties:
	\begin{itemize}
		\item\emph{Monotonicity}: \(S_h\) is increasing in \(g_h\) and \(g_h\) is non-decreasing in \(u\). Hence \(S_h\) is non-decreasing in \(u\). 
		
		\item\emph{Stability}: 
		The solution \(u_h\) of \(S_h=0\) is explicitly given in \eqref{eq:schemesol}. Denote in the following
		\[
		B_{Ch}^-(\bm x) := \bigl\{\bm y \in I \cap B_{Ch}(\bm x) : \bm y \le \bm x\bigr\}.
		\]
		Due to \cref{asm:g:monotone,asm:g:constant,asm:g:localised}, we have that
		\begin{multline*}
			\min_{\bm y\in B_{Ch}^-(\bm x)}\hspace{-.1cm} u_h(\bm y)
			\leq g_h\Bigl(\bm x,\bm\alpha,\min_{\bm y\in B_{Ch}^-(\bm x)} \hspace{-.1cm} u_h(\bm y)\Bigr)
			\\
			\leq g_h(\bm x,\bm\alpha,u_h)
			\leq g_h\Bigl(\bm x,\bm\alpha,\max_{\bm y\in B_{Ch}^-(\bm x)} \hspace{-.1cm} u_h(\bm y)\Bigr)
			= \max_{\bm y\in B_{Ch}^-(\bm x)}\hspace{-.1cm} u_h(\bm y)
		\end{multline*}
		for every \(h>0\).
		Inserting these estimates into \eqref{eq:schemesol}, we obtain that
		\[
			\min_{\bm y\in B_{Ch}^-(x)}\hspace{-.1cm} u_h(\bm y)
				\leq u_h(\bm x) 
				\leq \max_{\bm y\in B_{Ch}^-(x)}u_h(\bm y) + h\|f\|_\infty A_{max}
		\]
		with \(A_{max} = \max_{\bm\alpha\in A} \sqrt{\alpha_1\alpha_2}\). If we now consider \([0,1]_h^2\) as an directed acyclic graph, each path from \(\to \bm0\) to \(\bm x\in[0,1]_h^2\) has length of at most \(2N\). This gives that \(0\leq u_h(\bm x) \leq 2\|f\|_\infty A_{max}\).
	
		\item\emph{Consistency}: Since \(A\) is compact, we have that
		\begin{multline*}
			\smash{\lim_{\subalign{h&\to0\\\bm y&\to\bm x\\\xi&\to0}}}
			S_h\big(\bm y,\psi(\bm y)+\xi,\psi+\xi\big)\\
				\begin{aligned}
				&= 
					\lim_{\subalign{h&\to0\\\bm y&\to\bm x\\\xi&\to0}}
					\max_{\bm\alpha\in A} 
					\frac{g_h(\bm y,\bm\alpha,\psi+\xi) - \psi(\bm y)-\xi}{h} + 
					f(\bm y)\sqrt{\alpha_1\alpha_2} \\
				&= 
					\lim_{\subalign{h&\to0\\\bm y&\to\bm x\\\xi&\to0}}
					\max_{\bm\alpha\in A} 
					\frac{\psi(\bm y-h\bm\alpha) - \psi(\bm y)}{h} + 
					\omega_{\psi,L}(h) +
					f(\bm y)\sqrt{\alpha_1\alpha_2} \\
				&\stackrel{\mathmakebox[\widthof{=}]{\mathrm{(*)}}}{=}
					\max_{\bm\alpha\in A} 
					\lim_{\subalign{h&\to0\\\bm y&\to\bm x\\\xi&\to0}}
					\frac{\psi(\bm y-h\bm\alpha) - \psi(\bm y)}{h} + 
					\omega_{\psi,L}(h) +
					f(\bm y)\sqrt{\alpha_1\alpha_2} \\
				&= \max_{\bm\alpha\in A} D\psi(\bm x)(\bm\alpha) + f(\bm x)\sqrt{\alpha_1\alpha_2} \\
				&= H(\bm x,D\psi(\bm x)).
				\end{aligned}
		\end{multline*}
	In \((*)\), we used that the functional is uniformly continuous in \(\bm y, \bm\alpha,h,\xi\) to exchange the limit and maximisation. 	
	\end{itemize}
	Due to \cite[Theorem 2.1]{barles1991}, this proves convergence. 
\end{proof}

\subsection{Schemes based on \texorpdfstring{\boldmath\(D(u^2)\)}{D(u2)}}
\label{sec:implicit}

Recall that \(u\) is only H\"older continuous with exponent \(\frac12\) while \(u^2\) is Lipschitz continuous. This means that one might expect more accurate schemes based on an approximation of \(u^2\) rather than \(u\). This is done in \cite{calder2017}, where schemes are constructed for \(v := u^2\).\footnote{Additionally, a scheme for the Lipschitz continuous term \(w := u/\sqrt{x_1x_2}\) is also constructed in \cite{calder2017}. However, we deem this idea ill-suited for our approach due to the lack of simple closed-form expressions.}

The idea is to utilise that \(D(u^2) = 2uDu\), meaning that \(Du = D(u^2)/2u\) wherever \(u\geq 0\). In such, one would expect \(v\) to be a viscosity solution of 
\[
	\max_{\bm\alpha\in A} \frac{D(u(\bm x)^2)(\bm\alpha)}{2u(\bm x)} + f(\bm x)\sqrt{\alpha_1\alpha_2} = 0. 
\]
Already, this equation has problems with the singularity at \(u(\bm x) = 0\). However, we proceed by assuming for now that \(u(\bm x)>0\). We can then follow the above idea and construct schemes for \(v\) based on the approximation
\begin{equation}
		\label{eq:scheme2}
	\max_{\bm\alpha\in A} 
		\frac{g_h(\bm x,\bm\alpha,u)^2 - u(\bm x)^2}{2u(\bm x) h} + 
		f(\bm x)\sqrt{\alpha_1\alpha_2} = 0.
\end{equation}
For fixed \(\bm\alpha\), this can be modified into a quadratic equation in \(u(\bm x)\), meaning that schemes of this form will have multiple solutions. However, similar to the schemes in \cite{calder2017}, we are only interested in the largest of the solutions. 

Immediately, this means that \(u(\bm x) \geq hf(\bm x)\sqrt{\alpha_1\alpha_2}\). Moreover, we can multiply \eqref{eq:scheme2} with \(2u(\bm x)h\) to find that
\begin{multline*}
	\max_{\bm\alpha\in A} \Bigl[
		g_h(\bm x,\bm\alpha,u)^2 - u(\bm x)^2 + 
		2u(\bm x) hf(\bm x)\sqrt{\alpha_1\alpha_2}\Bigr] \\
	= \max_{\bm\alpha\in A} \Bigl[
		- \Bigl(\underbrace{u(\bm x)^2 -hf(\bm x)\sqrt{\alpha_1\alpha_2}}_{=:F(\bm\alpha)} \Bigr)^2
		+ \underbrace{h^2f(\bm x)^2\alpha_1\alpha_2
		+ g_h(\bm x,\bm\alpha,u)^2}_{=:G^2(\bm\alpha)}\Bigr] \\
	= \max_{\bm\alpha\in A} \bigl[-F(\bm\alpha)^2+G(\bm\alpha)^2\bigr] = 0
\end{multline*}
with \(F,G\geq0\). Observe that we have that
\[
	-F^2(\bm\alpha) + G^2(\bm\alpha) = (F(\bm\alpha)+G(\bm\alpha))(-F(\bm\alpha)+G(\bm\alpha))\leq 0
\]
for all \(\bm\alpha\in A\) with equality if and only if \(\bm\alpha\) is optimal. Moreover, equality can only be achieved if \(-F(\bm\alpha)+G(\bm\alpha) = 0\). Accordingly, the above scheme is identical to
\begin{multline*}
	\max_{\bm\alpha\in A} -F(\bm\alpha) + G(\bm\alpha) \\ =
	\max_{\bm\alpha\in A} 
		- u(\bm x) + hf(\bm x)\sqrt{\alpha_1\alpha_2}
		+ \sqrt{h^2f(\bm x)^2\alpha_1\alpha_2 + g_h(\bm x,\bm\alpha,u)^2} = 0,
\end{multline*}
which gives the closed form expression
\begin{equation}
		\label{eq:schemesol2}
		u(\bm x) = 
			\max_{\bm\alpha\in A} 
			hf(\bm x)\sqrt{\alpha_1\alpha_2}
			+ \sqrt{h^2f(\bm x)^2\alpha_1\alpha_2 + g_h(\bm x,\bm\alpha,u)^2}.
\end{equation}

\begin{theorem}
	\label{thm:scheme2}
	Under \cref{asm:a,asm:g}, the scheme \eqref{eq:schemesol2} is convergent. 
\end{theorem}
\begin{proof}
	Consider
	\[
		\overline u(\bm x) = \limsup_{\substack{\bm x\to\bm y\\h\to0}} u_h(\bm y), \quad
		\underline u(\bm x) = \liminf_{\substack{\bm x\to\bm y\\h\to0}} u_h(\bm y).
	\]
	The proof of \cite[Theorem 2.1]{barles1991} still holds for all \(\bm x\) for which \(\underline u(\bm x)>0\), and proving the properties of monotonicity, stability and consistency is similar to that of \cref{thm:scheme}. Wherever \(\underline u(\bm x)=0\), the singularity of the scheme breaks the proof of convergence. However, we have that \(u(\bm x)=0\) if and only if \(f(\bm y) = 0\) for all \(\bm y\leq \bm x\). Accordingly, it is sufficient to prove that this property holds for \(\underline u\) and \(\overline u\) as well.
	
	Observe that
	\begin{align*}
		u_h(\bm x) 
			&=
			\max_{\bm\alpha\in A} 
			hf(\bm x)\sqrt{\alpha_1\alpha_2}
			+ \sqrt{h^2f(\bm x)^2\alpha_1\alpha_2 + g_h(\bm x,\bm\alpha,u_h)^2}
			\\
			&\geq 
			\max_{\bm\alpha\in A} 
			hf(\bm x)\sqrt{\alpha_1\alpha_2}
			+ g_h(\bm x,\bm\alpha,u_h).
	\end{align*}
	Inductively, this gives that \(u_h\geq \tilde u_h\) where \(\tilde u_h\) is the solution of \eqref{eq:schemesol} for the same choices of \(A\) and \(g_h\). Since \(\tilde u_h\) is convergent, this implies that \(u_h(\bm x)\geq\tilde u_h(\bm x)>0\) wherever \(f(\bm x) > 0\).

	Now assume that \(f(\bm y) = 0\) for all \(\bm y\leq \bm x\). Then, since \(g_h(\cdot,\cdot,0) = 0\), it is clear that \(u_h(\bm x) = \tilde u_h(\bm x) = 0\), concluding the proof.

\end{proof}

\section{Proposed Schemes}\label{se:proposed}
Now, it remains to choose \(A\) and \(g_h\) such that we obtain efficient schemes. In particular, we desire closed form expressions for both \(u_h\) and the optimal \(\bm\alpha\) used in each step.

For the closed form expressions for the schemes, it is useful to denote \(\bm x_1^1\) as the current grid point for which we are solving the schemes. In addition, we denote \(\bm x_0^1\), \(\bm x_1^0\) and \(\bm x_0^0\) as the other three points of the grid cell. Moreover, denote \(u_i^j = u(\bm x_i^j)\) and \(f_i^j = f(\bm x_i^j)\). Until now, we have assumed that \(f_i^j\) is evaluated exactly. For approximations of this term, see \cref{sec:appf}.

\subsection{Schemes based on \texorpdfstring{\boldmath\(Du\)}{Du}}
For the scheme \eqref{eq:schemesol}, we start by letting \(g_h\) be the linear interpolation of \(u\) through the points \(\bm x_0^1\), \(\bm x_1^0\) and \(\bm x_0^0\), which reads 
\begin{equation}
\label{eq:g1}
\begin{aligned}
	g_h 
		&= u_0^0 + (1-\alpha_1)(u_1^0-u_0^0)  + (1-\alpha_2)(u_0^1-u_0^0) \\
		&= (\alpha_1+\alpha_2-1)u_0^0 + (1-\alpha_1)u_1^0  + (1-\alpha_2)u_0^1.
\end{aligned}
\end{equation}
It is easy to check that this satisfies all properties of \cref{asm:g} apart from monotonicity. To ensure monotonicity, it is required that \(0\leq\alpha_1,\alpha_2\leq1\) and that \(\alpha_1+\alpha_2\geq 1\). This holds for every \(\alpha \in A_r := \{\bm\alpha\in\mathbb R_{\geq0}^2\mid |\bm\alpha|_r=1\}\) for \(1\leq r\leq +\infty\). We find that for the choices \(A = A_1\) and \(A = A_\infty\), we can solve the schemes analytically.

Choosing \(A = A_1\), combined with the first approximation \eqref{eq:g1}, we obtain the following solution for the scheme, and the optimal \(\bm\alpha^*\):
\begin{equation}
\label{eq:u1}
\tag{U\textsubscript{\(1\)}}
\begin{aligned}
	hS_h &= u_1^1 - \frac12\left(
		u_0^1+u_1^0+\sqrt{
			(u_0^1-u_1^0)^2 + (hf_1^1)^2
		}
	\right), \\
	u_1^1 &= \frac12\left(
		u_0^1+u_1^0+\sqrt{
			(u_0^1-u_1^0)^2 + (hf_1^1)^2
		}
	\right), \\
	\bm\alpha^* &=\biggl(
		\frac12\Bigl(1+\frac{u_0^1-u_1^0}{(u_0^1-u_1^0)^2+(hf_1^1)^2}\Bigr), \,
		\frac12\Bigl(1-\frac{u_0^1-u_1^0}{(u_0^1-u_1^0)^2+(hf_1^1)^2}\Bigr)
	\biggr).
\end{aligned}
\end{equation}
Interestingly, this is exactly the original scheme proposed in \cite{calder2013}. Using \(A = A_\infty\), we obtain with the abbreviation
\[
u^*_* := \max\{u_0^1, u_1^0\}
\]
that
\begin{equation}
\label{eq:uinf}
\tag{U\textsubscript{\(\infty\)}}
\begin{aligned}
	hS_h &= u_1^1 - \begin{cases}\displaystyle
		u_*^* + \frac{(hf_1^1)^2}{4(u_*^*-u_0^0)}, & 2(u_*^*-u_0^0) > \sqrt{(hf_1^1)^2}, \\
		u_0^0 + hf_1^1, & \text{otherwise},
	\end{cases} \\
	u_1^1 &= \begin{cases}\displaystyle
		u_*^* + \frac{(hf_1^1)^2}{4(u_*^*-u_0^0)}, & 2(u_*^*-u_0^0) > \sqrt{(hf_1^1)^2}, \\
		u_0^0 + hf_1^1, & \text{otherwise},
	\end{cases} \\
	\bm\alpha^* &= \begin{cases}
		\displaystyle
		\biggl(1,\, \frac{hf_1^1}{2(u_0^1-u_0^0)}\biggr), & u_0^1\geq u_1^0, 2(u_0^1-u_0^0) > \sqrt{(hf_1^1)^2},\\
		\displaystyle
		\biggl(\frac{hf_1^1}{2(u_1^0-u_0^0)},\, 1\biggr), & u_1^0> u_0^1, 2(u_1^0-u_0^0) > \sqrt{(hf_1^1)^2},\\
		(1,\,1), &\text{otherwise}.
	\end{cases}
\end{aligned}
\end{equation}

\subsection{Schemes based on \texorpdfstring{\boldmath\(D(u^2)\)}{D(u2)}}
For the scheme \eqref{eq:schemesol}, it is useful to express the schemes for \(v_i^j := (u_i^j)^2\). Instead of linearly interpolating \(u\), we linearly interpolate \(v\), meaning that 
\begin{equation}
\label{eq:g2}
\begin{aligned}
	g_h^2
		&= v_0^0 + (1-\alpha_1)(v_1^0-v_0^0)  + (1-\alpha_2)(v_0^1-v_0^0) \\
		&= (\alpha_1+\alpha_2-1)v_0^0 + (1-\alpha_1)v_1^0  + (1-\alpha_2)v_0^1.
\end{aligned}
\end{equation}
Here, we have the same conditions for monotonicity as for the schemes based on \(Du\), and we have analytical solutions for \(A = A_1\):
\begin{equation}
\label{eq:v1}
\tag{V\textsubscript{\(1\)}}
\begin{aligned}
	2h&\sqrt{v_1^1}S_h = v_1^1 - \frac12\left(
		v_0^1+v_1^0+\sqrt{
			(v_0^1-v_1^0)^2 + (hf_1^1)^2
		}
	\right), \\
	v_1^1 &= \frac12\left(
		v_0^1+v_1^0+h^2f^2
		+
		\sqrt{
			(v_0^1-v_1^0)^2 + 2(v_0^1-v_1^0)(hf_1^1)^2 + (hf_1^1)^4
		}
	\right), \\
	\bm\alpha^* &=\biggl(
		\frac12\Bigl(1+\frac{v_0^1-u_1^0}{(v_0^1-v_1^0)^2+4v_1^1(hf_1^1)^2}\Bigr), \,
		\frac12\Bigl(1-\frac{v_0^1-u_1^0}{(v_0^1-v_1^0)^2+4v_1^1(hf_1^1)^2}\Bigr)
	\biggr).
\end{aligned}
\end{equation}
Again, we have that the scheme using \(A_1\) is identical to that of \cite{calder2017}. For \(A=A_\infty\), we have with
\[
v_*^* := \max\{v_0^1,v_1^0\}
\]
that
\begin{equation}
\tag{V\textsubscript{\(\infty\)}}
\label{eq:vinf}
\begin{aligned}
	2h\sqrt{v_1^1}S_h &= v_1^1 - \begin{cases}\displaystyle
		v_*^* + \frac{v_1^1(hf_1^1)^2}{v_*^*-v_0^0}, & v_*^*-v_0^0 > \sqrt{v_1^1}hf_1^1, \\
		v_0^0 + hf_1^1, & \text{otherwise},
	\end{cases} \\
	v_1^1 &= \begin{cases}\displaystyle
		\frac{v_*^*(v_*^*-v_0^0)}{v_*^*-v_0^0-(hf_1^1)^2} & (v_*^*-v_0^0)(v_*^*-v_0^0-(hf_1^1)^2) > v_*^*hf_1^1, \\
		u_0^0 + hf_1^1, & \text{otherwise},
	\end{cases} \\
	\bm\alpha^*
	&= \begin{cases}
		\displaystyle
		\biggl(1,\, \frac{\sqrt{v_1^1}hf_1^1}{v_0^1-v_0^0}\biggr), & v_0^1\geq v_1^0,\, v_0^1-v_0^0 > \sqrt{v_1^1}hf_1^1,\\
		\displaystyle
		\biggl(\frac{\sqrt{v_1^1}hf_1^1}{v_1^0-v_0^0},\, 1\biggr), & v_1^0> v_0^1,\, v_1^0-v_0^0 > \sqrt{v_1^1}hf_1^1,\\
		(1,\,1), &\text{otherwise}.
	\end{cases}
\end{aligned}
\end{equation}
Due to the singularity at \(v_1^1 = 0\), we have that \(S_h\) is not defined in these cases. However, the solution of the schemes still exist.

\subsection{Higher Order Filtered Schemes}
It is known that one cannot construct higher order schemes for solving HJB equations, as one requires monotone schemes to obtain convergence. Still, it is common to construct so-called \emph{filtered} schemes. These schemes are based on a high-order (possibly non-monotone) scheme \(S_h^a\), and a monotone scheme \(S_h^m\). The idea is to choose the higher order scheme only if its approximation to the Hamiltonian is sufficiently close to that of the monotone scheme. The selection criterion is typically chosen as \(|S_h^a - S_h^m| \leq k\sqrt{h}\) for some constant \(k\) to preserve the theoretical \(\sqrt{h}\) convergence which is typical for schemes for HJB equations. Therefore, we define the filtered scheme as
\[
	S_h^f := \begin{cases}
		S_h^a, & |S_h^a - S_h^m| \leq k\sqrt{h}, \\
		S_h^m, & |S_h^a - S_h^m| > k\sqrt{h}.
	\end{cases}
\]
This can be implemented by first solving \(S_h^a = 0\). If this solutions satisfies \(|S_h^m| \leq k\sqrt h\), we keep the solution. Otherwise, we use the solution of \(S_h^m = 0\).

To construct the high order scheme, we use the same idea as in \cref{sec:schemes}, except that we use central differences, rather than backward differences. In practice, this means that we approximate
\[
	-Du(\bm x)\bm\alpha = \frac{u(\bm x-\frac h2\bm\alpha) - u(\bm x+\frac h2\bm\alpha)}{h}
\]
with a similar approximation to \(D(u^2)\). Approximating \(u(\bm x-\frac h2\bm\alpha)\) and \(u(\bm x+\frac h2\bm\alpha)\) using \(g_{\frac h2}\), we obtain the two general schemes:
\begin{align}
	\label{eq:scheme2nd}
	S_h^a &= 
	\max_{\alpha\in A} 
		\frac{g_{\frac h2}(\bm x,\bm\alpha,u) - g_{\frac h2}(\bm x,-\bm\alpha,u)}{h} + 
		f(x)\sqrt{\alpha_1\alpha_2} = 0 \\
	\label{eq:scheme22nd}
	S_h^a &= 
	\max_{\alpha\in A} 
		\frac{g_{\frac h2}(\bm x,\bm\alpha,u)^2 - g_{\frac h2}(\bm x,-\bm\alpha,u)^2}{2u(\bm x)h} + 
		f(x)\sqrt{\alpha_1\alpha_2} = 0.
\end{align}
These schemes will be solved with \(\bm x\) being the centre of a grid cell. 

In \eqref{eq:scheme2nd}, we approximate \(\smash{g_{\frac h2}}(\bm x,\bm\alpha,u)\) using a linear interpolation of \(u_0^1\), \(u_1^0\) and \(u_0^0\) for positive \(\bm\alpha\) and a linear interpolation of \(u_0^1\), \(u_1^0\) and \(u_1^1\) for negative \(\bm\alpha\). Interestingly, this gives a scheme which is independent of \(A\), reading
\begin{align*}
	u_1^1 &= u_0^0+\sqrt{
			(u_0^1-u_1^0)^2 + (hf_1^1)^2
		}, \\
	\bm\alpha^* &=\biggl(
		\frac12\Bigl(1+\frac{u_0^1-u_1^0}{(u_0^1-u_1^0)^2+(hf_1^1)^2}\Bigr), \,
		\frac12\Bigl(1-\frac{u_0^1-u_1^0}{(u_0^1-u_1^0)^2+(hf_1^1)^2}\Bigr)
	\biggr).
\end{align*}
With similar linear approximations in \eqref{eq:scheme22nd}, we need an approximation of \(u(\bm x)\), present in the denominator, since \(\bm x\) does not coincide with a grid cell. Using \(2u(\bm x) \approx u_0^1+u_1^0\), we obtain
\begin{align*}
	v_1^1 &= 
		v_0^0+\frac12h^2f^2
		+
		\sqrt{
			(v_0^1-v_1^0)^2 + (2v_0^0+v_0^1+v_1^0)(hf_1^1)^2 + \frac14(hf_1^1)^4
		}, \\
	\bm\alpha^* &=\biggl(
		\frac12\Bigl(1+\frac{v_0^1-u_1^0}{(v_0^1-v_1^0)^2+4v_1^1(hf_1^1)^2}\Bigr), \,
		\frac12\Bigl(1-\frac{v_0^1-u_1^0}{(v_0^1-v_1^0)^2+4v_1^1(hf_1^1)^2}\Bigr)
	\biggr).
\end{align*}

\subsection{Fully Discretised Schemes}
The scheme \eqref{eq:scheme} can in fact be used to formulate fully discretised schemes with some modification. We start by replacing \(A\) with a variable admissible space  \(A = A_h \subset \mathbb N_0^2 \setminus \bm0\), i.e. the set of pairs of non-negative integers. With this choice, \(\bm x-h\bm \alpha\) coincides with other grid points of the square grid (as long as \(\bm x\) lies on a grid point). In such, the scheme
\begin{equation}
	\label{eq:scheme3sol}
	u(\bm x) = \max_{\bm\alpha\in A_h} u(\bm x-h\bm\alpha) + f(\bm x)\sqrt{\alpha_1\alpha_2}
\end{equation}
can be solved without the need for an approximation to \(u(\bm x-h\bm\alpha)\). By setting \(\bm x=\bm x_i^j\in[0,1]_h^2\), this scheme reads
\begin{equation}
	\label{eq:ddp}
	\tag{DDP}
	u_i^j = \max_{(k,l)\in A_h} u_{i-k}^{j-l} + 
		hf_i^j\sqrt{kl}.
\end{equation}
This is exactly the discretised dynamic programming method commonly used in the literature. Under certain assumptions on \(A_h\), we can still use the HJB based approach to prove convergence of this scheme. See \cref{asm:ah} and \cref{thm:scheme3} in Appendix \ref{sec:app} for details. 

Choosing \(A_h\) requires a compromise between accuracy and complexity. We want \(A_h\) to include as many directions as possible for optimal accuracy. However,  the larger the set \(A_h\), the higher the computational cost. 
One example of at set satisfying \cref{asm:ah} is
\[
	A_h = \{\bm\alpha\in\mathbb N_0^2 \mid |\bm\alpha| \leq kh^{-r}\}
\]
for constants \(k>0\) and \(0<r<1\). Here, \(r = \frac12\) will typically give a good compromise between accuracy and efficiency.


\section{Numerical Computation of Geodesics}\label{se:backtracking}

The numerical solution of the value function gives an estimate to \(u(\bm 1)\), which in turn can be used to approximate the geodesic distance through \(\dist^{\mathcal S}([c_1],\, [c_2]) = \arccos (u(\bm 1))\). Additionally, through a backtracking method, we can use the value function to obtain an estimate of the solution \(\bm\varphi\) of the variational problem \eqref{eq:prob}. This can then be used to estimate the shape space geodesic between \(c_1\) and \(c_2\). 

\subsection{Backtracking}
To retrieve the optimal reparametrisation path \(\bm\varphi\), we propose a piecewise constant interpolation of the maximiser \(\bm\alpha^*\) of the approximated HJB equation, where \(\bm\alpha^*\) is constant on each grid cell \((x_i,x_{i+1}]\times(x^j,x^{j+1}]\). With \(\bm\varphi'(t) = \bm\alpha^*\), this gives a first order piecewise constant differential equation for \(\bm\varphi'\), which therefore can be computed explicitly. 

In practice, the path \(\bm\varphi\) will be piecewise linear, only changing direction when intersecting a grid line, meaning that the path can be represented by a sequence \(\{\bm\varphi_k\}\) with length at most \(2N\).  Assume that the backtracking procedure has reached the point \(\bm\varphi_k \in (x_i,x_{i+1}]\times(x^j,x^{j+1}]\). In order to obtain the next point in the sequence, we construct the line \(\bm\psi(t) = \bm\varphi_k - t\bm\alpha^*\), defined for \(t\geq 0\) where \(\bm\alpha^*\) is optimal for the given grid cell. Then, we find the intersection point between \(\bm\psi\) and the vertical line \((x_i,\cdot)\) and the intersection point between \(\bm\psi\) and the horizontal line \((\cdot,x^j)\). The next point in the sequence will then be the maximum of these points. This reads
\[
	\bm\varphi_{k-1} = \max\bigl\{(x_i, \varphi_{2,k}-(\varphi_{1,k}-x_i)\alpha_2^*/\alpha_1^*),\,(\varphi_{1,k}-(\varphi_{2,k}-x^j)\alpha_1^*/\alpha_2^*, x^j)\bigr\}.
\]
Note that since the path \(\bm\varphi\) is monotone increasing, one of the intersection points will actually be maximal with respect to the standard partial ordering of \(\mathcal R^2\). 

The terminal condition for the path \(\bm\varphi\) is \(\bm\varphi(1) = \bm 1\), which also acts as the starting point for the backtracking procedure. With the convention that \(\bm\alpha^* = (1,0)\) wherever \(\bm x = (x_1,0)\) and \(\bm\alpha^*=(0,1)\) wherever \(\bm x = (0,x_2)\), we ensure that the (inferred) initial condition \(\bm\varphi(0) = \bm0\) is met.

\subsection{Computing Geodesics and Geodesic Distances}
Now that we have an estimate of \(\bm\varphi\), we can estimate the SRVTs after reparametrisation. Similar to the reparametrisation path, we construct a sequence of points of the form
\[
	q_{i,k} := q_i(\varphi_{i,k})\sqrt{\frac{\varphi_{i,k}-\varphi_{i,k-1}}{\Delta t_k}},
\]
for \(i = 1,\, 2\). Note that this expression requires \(\Delta t_k\), representing the joint parametrisation of \(\varphi_1\) and \(\varphi_2\). Since the problem is reparametrisation invariant, this can be chosen based on the application. One natural option is to choose \(\Delta t_k = \frac12(\varphi_{1,k}-\varphi_{1,k-1} + \varphi_{2,k}-\varphi_{2,k-1})\), motivated from the assumption that \(\|\bm\varphi'\|_1=1\). This constraint is especially useful since  \(\bm\varphi'\) is bounded and the domain \(I\) remains unchanged. 

Using the point estimates of the SRVTs, we can approximate the objective function and the geodesics. First of all, for the objective function, we have the following estimate:
\begin{align*}
	J_h(\bm\varphi_h)
		&= \sum_{k} \langle q_{1,k},\, q_{2,k} \rangle \Delta t_k \\
		&= \sum_{k} \langle q_1(\varphi_{1,k}),\, q_2(\varphi_{2,k})\rangle \sqrt{(\varphi_{k,1}-\varphi_{k-1,1})(\varphi_{k,2}-\varphi_{k-1,2})}.
\end{align*}
Observe in particular that this expression is independent of \(\Delta t_k\), as desired. 

Similarly, we can pointwise approximate the geodesic using
\begin{multline*}
	\gamma_k(\tau) = 
		w_h^{\mathcal S}(1-\tau)
		q_{1,k}
		+ w_h^{\mathcal S}(\tau)
		q_{2,k} \\
	=
		w_h^{\mathcal S}(1-\tau)
		q_1(\varphi_{1,k})\sqrt{\frac{\varphi_{1,k}-\varphi_{1,k-1}}{\Delta t_k}}
		+ w_h^{\mathcal S}(\tau)
		q_2(\varphi_{2,k})\sqrt{\frac{\varphi_{2,k}-\varphi_{2,k-1}}{\Delta t_k}},
\end{multline*}
where \(w_h^{\mathcal S}(\tau) = \sin(\tau\arccos J_h(\bm\varphi_h))/\sin(\arccos J_h(\varphi_h))\).
In the pre-shape space, the geodesic can be approximated using
\[
	Q^{-1}(\gamma(\tau))(t_k) \approx \sum_{l=1}^k \gamma_l(\tau)|\gamma_l(\tau)|\Delta t_l
	= \sum_{l=1}^k \gamma_l(\tau)\sqrt{\Delta t_l}\bigl|\gamma_l(\tau)\sqrt{\Delta t_l}\bigr|.
\]
Similarly to the objective function, this estimate is independent of \(\Delta t\), as desired.


\section{Numerical Experiments}\label{se:experiments}

To test the numerical schemes, we use three test problems labeled \ref{exp:a}, \ref{exp:b}, \ref{exp:c} for which the curves and shape space geodesics are illustrated in \cref{fig:geodesics}. For each of the test problems, the schemes were ran with grid sizes \(N = h^{-1} = 5\cdot 2^2, \ \ldots, \ 5\cdot 2^{10}\) (for the discretised dynamic programming, the smallest two step sizes were omitted due to computational complexity).

For \cref{exp:a,exp:b}, we use arc length parametrisation as the initial parametrisation of the curves. For these problems, we do not have analytic solutions for any of the variables of interest. The analytic solutions were therefore approximated using the filtered scheme with a fine grid size \(h^{-1} = \epsilon^{-1} = 5\cdot 2^{11}\). For \cref{exp:c}, we compare two curves with equal shape but different initial parametrisations. In particular, we let \(c_1 = c_0\circ\psi_1\) and \(c_2=c_0\circ\psi_2\) with \(c_0\) being the arc length parametrisation of the curve. Here, we use the M\"obius transformations
\begin{align*}
	\psi_1(t) = 3t/(1+2t), \quad
	\psi_2(t) = t/(3-2t),
\end{align*}
which are each other's inverses. Hence, one solution of the reparametrisation problem is given by \(\varphi_1 = \psi_1^{-1} = \psi_2\) and \(\varphi_2 = \psi_2^{-1} = \psi_1\). For this problem, we have the exact geodesic distance \(d(c_1,c_2) = 1\) and exact expressions for the geodesics (which are constant in \(\tau\)). 

It has been demonstrated in \cite{thawinrak2017} that filtered schemes can give an improvement for simple problems. However, for our experiments, we found that there was no significant improvement of the filtered schemes compared to the best performing of the monotone schemes.\footnote{There is in some cases a small improvement, but this is outweighed by the added computational time.} Therefore, we will only compare the four monotone schemes presented together with the fully discretised dynamic programming. For the discretised dynamic programming, we tested \(A_h = \{\bm\alpha\in\mathbb N_0^2 \mid |\bm\alpha| \leq kh^{-r}\}\) for different values of \(k,r\). We found that \(k=\frac 34\) and \(r=\frac12\) gave the best performance when measuring accuracy vs computation time. For all schemes, we found that the approximation of \(f\) as described in \cref{sec:appf} gives better results. These approximations were therefore used in the following experiments. 

\begin{figure}
	\begin{experiment}
		\item \label{exp:a}
		\begin{minipage}{10.75cm}\centering

\begin{tikzpicture}[trim axis left, trim axis right,baseline=-0.1cm]
	\begin{axis}[
		%
		%
		at = {(0,0)},
		width = 7cm,
		height = 5cm,
		%
		%
		clip=false,
		hide axis,
		axis equal,
		enlargelimits=false,
		y dir=reverse,
	]
		\addplot [geodesic] table [x index=1, y index=2, meta index=0] {data/ex1/parametrisation.txt};
		\addplot [geodesic] table [x index=3, y index=4, meta index=0] {data/ex1/parametrisation.txt};
		\foreach \x/\y in {1/8,2/9,3/10,4/11,5/12,6/13,7/14} {
			\addplot [geodesic] table [x index = \x, y index = \y, meta index=0] {data/ex1/geodesics.txt};
		}
	\end{axis}
\end{tikzpicture}\end{minipage}
	\vspace{.1cm}
		\item \label{exp:b}
		\begin{minipage}{10.75cm}\centering

\begin{tikzpicture}[trim axis left, trim axis right,baseline=-0.1cm]
	\begin{axis}[
		%
		%
		at = {(0,0)},
		width = 7cm,
		height = 5cm,
		%
		%
		clip=false,
		hide axis,
		axis equal,
		enlargelimits=false,
		y dir=reverse,
	]
		\addplot [geodesic] table [x index=1, y index=2, meta index=0] {data/ex2/parametrisation.txt};
		\addplot [geodesic] table [x index=3, y index=4, meta index=0] {data/ex2/parametrisation.txt};
		\foreach \x/\y in {1/8,2/9,3/10,4/11,5/12,6/13,7/14} {
			\addplot [geodesic] table [x index = \x, y index = \y, meta index=0] {data/ex2/geodesics.txt};
		}
	\end{axis}
\end{tikzpicture}\end{minipage}
	\vspace{.1cm}
		\item \label{exp:c}
		\begin{minipage}{10.75cm}\centering

\begin{tikzpicture}[trim axis left, trim axis right,baseline=-0.1cm]
	\begin{axis}[
		%
		%
		at = {(0,0)},
		width = 7cm,
		height = 5cm,
		%
		%
		clip=false,
		hide axis,
		axis equal,
		enlargelimits=false,
		y dir=reverse,
	]
		\addplot [geodesic] table [x index=1, y index=2, meta index=0] {data/ex3/parametrisation.txt};
		\addplot [geodesic] table [x index=3, y index=4, meta index=0] {data/ex3/parametrisation.txt};
		\foreach \x/\y in {1/8,2/9,3/10,4/11,5/12,6/13,7/14} {
			\addplot [geodesic] table [x index = \x, y index = \y, meta index=0] {data/ex3/geodesics.txt};
		}
	\end{axis}
\end{tikzpicture}\end{minipage}
	\end{experiment}
	\caption{Left: curves coloured by initial parametrisation. Right: shape space geodesics.}
	\label{fig:geodesics}
\end{figure}
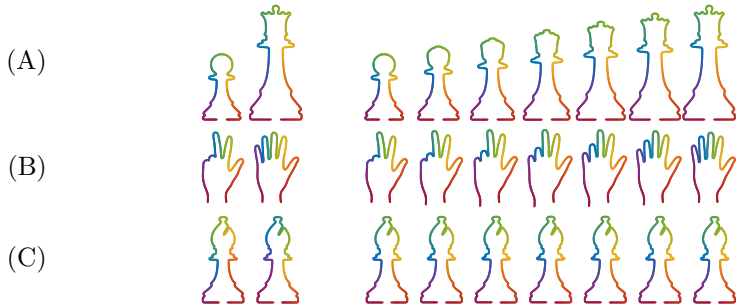

\subsection{Presence of Local Solutions}\label{se:localsols}
Dynamic programming based methods typically converge slower than gradient based method. Therefore, it is important to assess whether local solutions are present or not. In order to do this, we consider the total value function
\[
	u_{tot}(\bm x) :=
	\sup_{\bm \varphi\in \mathcal A}  \int_0^1 f (\varphi_1,\varphi_2) \sqrt{\smash{\varphi_1'\varphi_2'}\vphantom f}dt \quad\text{s.t.}\quad \bm\varphi(\tfrac12) = \bm x.
\]
This variant of the value function measures the similarity between the curves \(c_1,\,c_2\) given a landmark constraint at \(\bm x\), that is, requiring that the point \(c_1(x_1)\) is registered to \(c_2(x_2)\). If \(u_{tot}\) has a local maximum at \(\bm x\), there is a local solution of \eqref{eq:prob} passing through \(\bm x\). This means that \(u_{tot}\) can be used to find local solutions. The total value function will not characterise all local solutions, but the number of local maxima of \(u_{tot}\) is an indication of the number of local solutions of \eqref{eq:prob}. Note that all maxima of \(u_{tot}\) are inherently flat, meaning that there are in practice paths of local maxima. 

The total value function is easy to compute. The ``standard'' value function \eqref{eq:value} was defined by maximising over all paths from \(\bm0\) to \(\bm x\). Alternatively, we can define a reversed value function where we optimise over all paths from \(\bm x\) to \(\bm 1\). Since the problem is fundamentally invariant to reparametrisations, these are identical problems up to replacing \(f(x_1,x_2)\) with \(f(1-x_1,1-x_2)\). Then, the sum of the standard and reversed value functions together gives the total value function. For each local maximum of \(u_{tot}\), one can run the backtracking algorithm in both directions to obtain a local solution of \eqref{eq:prob}. 

\begin{figure}
	\centering
	\begin{subfigure}[b]{.3\textwidth}
		\centering

\begin{tikzpicture}[trim axis left, trim axis right]

	\begin{axis}[
		scale only axis,
		height = .2cm,
		width = 2.5cm,
		at = {(-1.25cm,1.45cm)},
		ymin = 0,
		ymax = 1,
		xmin = 0,
		xmax = .9293,
		ytick = \empty,
		xtick={0,0.25,0.5,0.75,.9293},
		xlabel near ticks, 
		xtick pos=right,
		xticklabel style = {rotate=45},
		anchor=west,
		major tick length=2pt,
		axis on top,
	]
		\addplot [on layer=axis background] graphics [xmin=0,ymin=0,ymax=1,xmax=.9293] {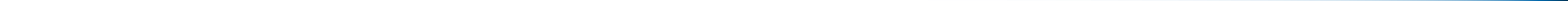};
	\end{axis}
	\begin{axis}[
		scale only axis,
		axis equal,
		height = 2.5cm,
		width = 2.5cm,
		at = {(0cm,0cm)},
		xmin=0,
		xmax=1,
		ymin=0,
		ymax=1,
		xtick=\empty,
		ytick=\empty,
		xlabel={$x_1\vphantom{\varphi_1(t)}$},
		ylabel={$x_2$},
		ylabel style = {rotate=-90},
		axis on top,
	]
		\addplot [on layer=axis background] graphics [xmin=0,ymin=0,xmax=1,ymax=1] {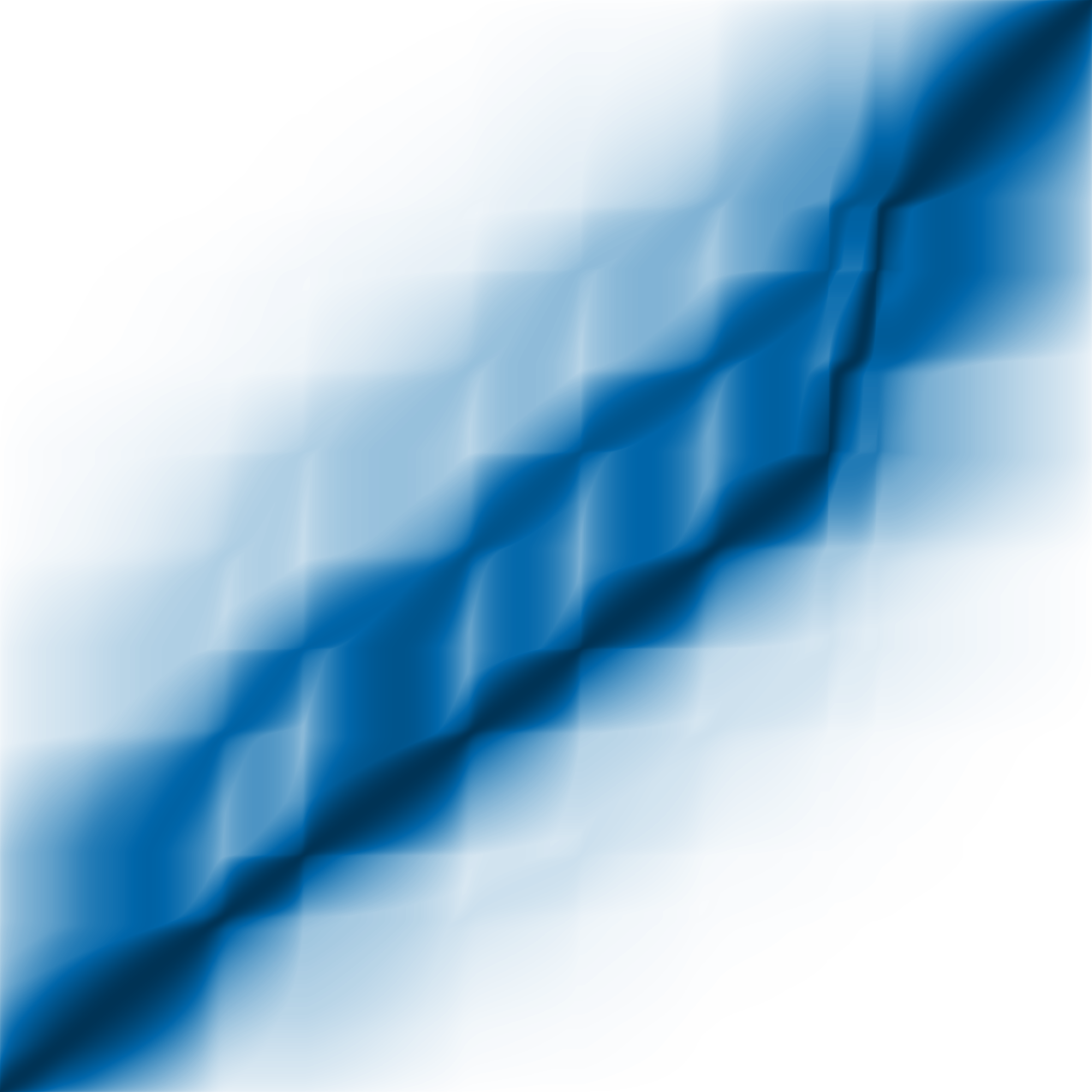};
	\end{axis}

\end{tikzpicture}
		\caption{\, }
		\label{fig:value:u}
	\end{subfigure}
	\begin{subfigure}[b]{.3\textwidth}
		\centering

\begin{tikzpicture}[trim axis left, trim axis right]

	\begin{axis}[
		scale only axis,
		axis equal,
		height = 2.5cm,
		width = 2.5cm,
		at = {(5cm,0cm)},
		xmin=0,
		xmax=1,
		ymin=0,
		ymax=1,
		xtick=\empty,
		ytick=\empty,
		xlabel={$\varphi_1(t)$},
		ylabel={$\varphi_2(t)$},
		ylabel style = {rotate=-90},
		axis on top,
		cycle list name=local lines,
	]
		\foreach \i in {1,...,27} {
			\addplot table {data/value/x\i.txt};
		}
	\end{axis}

\end{tikzpicture}
		\caption{\,}
		\label{fig:value:umax}
	\end{subfigure}
	\begin{subfigure}[b]{.3\textwidth}
		\centering

\begin{tikzpicture}[trim axis left, trim axis right]

	\begin{axis}[
		scale only axis,
		axis equal,
		height = 2.5cm,
		width = 2.5cm,
		at = {(5cm,0cm)},
		xmin=0,
		xmax=1,
		ymin=0,
		ymax=1,
		xtick=\empty,
		ytick=\empty,
		xlabel={$\varphi_1(t)$},
		ylabel={$\varphi_2(t)$},
		ylabel style = {rotate=-90},
		axis on top,
		cycle list name=local lines,
	]
		\addplot table {data/value/x15.txt};
		\addplot table {data/value/x5.txt};
		\addplot table {data/value/x1.txt};
	\end{axis}

\end{tikzpicture}
		\caption{\,}
		\label{fig:value:umax2}
	\end{subfigure}
	\caption{The total value function (a), some local maxima of \eqref{eq:prob} (b), and three local maxima near the diagonal (c).}
	\label{fig:value}
\end{figure}
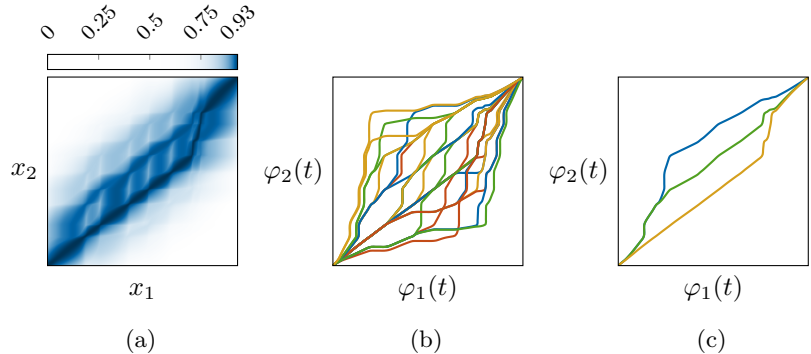
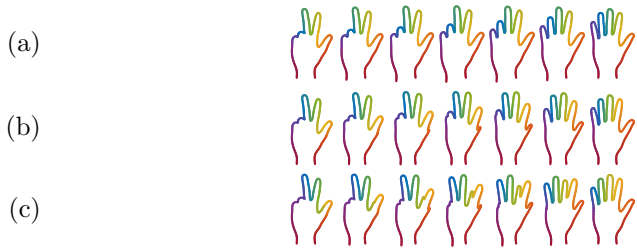
\begin{figure}
	\begin{enumerate}[label*=(\alph*)]
		\item
		\begin{minipage}{10.75cm}\centering

\begin{tikzpicture}[trim axis left, trim axis right,baseline=-0.1cm]
	\begin{axis}[
		%
		%
		at = {(0,0)},
		width = 4.5cm,
		height = 5cm,
		%
		%
		clip=false,
		hide axis,
		axis equal,
		enlargelimits=false,
		y dir=reverse,
	]
		\foreach \x/\y in {1/8,2/9,3/10,4/11,5/12,6/13,7/14} {
			\addplot [geodesic] table [x index = \x, y index = \y, meta index = 0] {data/ex2/geodesics1.txt};
		}
	\end{axis}
\end{tikzpicture}\end{minipage}
	\vspace{.1cm}
		\item 
		\begin{minipage}{10.75cm}\centering

\begin{tikzpicture}[trim axis left, trim axis right,baseline=-0.1cm]
	\begin{axis}[
		%
		%
		at = {(0,0)},
		width = 4.5cm,
		height = 5cm,
		%
		%
		clip=false,
		hide axis,
		axis equal,
		enlargelimits=false,
		y dir=reverse,
	]
		\foreach \x/\y in {1/8,2/9,3/10,4/11,5/12,6/13,7/14} {
			\addplot [geodesic] table [x index = \x, y index = \y, meta index = 0] {data/ex2/geodesics5.txt};
		}
	\end{axis}
\end{tikzpicture}\end{minipage}
	\vspace{.1cm}
		\item
		\begin{minipage}{10.75cm}\centering

\begin{tikzpicture}[trim axis left, trim axis right,baseline=-0.1cm]
	\begin{axis}[
		%
		%
		at = {(0,0)},
		width = 4.5cm,
		height = 5cm,
		%
		%
		clip=false,
		hide axis,
		axis equal,
		enlargelimits=false,
		y dir=reverse,
	]
		\foreach \x/\y in {1/8,2/9,3/10,4/11,5/12,6/13,7/14} {
			\addplot [geodesic] table [x index = \x, y index = \y, meta index = 0] {data/ex2/geodesics15.txt};
		}
	\end{axis}
\end{tikzpicture}\end{minipage}
	\end{enumerate}
	\caption{Pre-shape geodesics with parametrisation corresponding to the three local maxima visualised in \cref{fig:value:umax2}.}
	\label{fig:geodesics2}
\end{figure}

For \cref{exp:b}, the total value function was estimated using \(h=5\cdot10^{-4}\). The estimate is visualised in \cref{fig:value:u} and the local maxima of \eqref{eq:prob} are visualised in \cref{fig:value:u}. Note that a highly nonlinear colormap has been used in \cref{fig:value:u} to accentuate the local maxima. Using this method, 27 local maxima were found. However, this method of finding local maxima is conservative, and there are likely a lot more. We chose three local maxima close to the diagonal, and computed the resulting curve space geodesics. The result is visualised in \cref{fig:geodesics2}. As one can see, the resulting geodesics are very different. A priori, it is hard to tell which one of these solutions a local, gradient based methods will find. This accentuates the importance of global solvers.

\subsection{Convergence of the Value Function}
For the value function, we have theoretical point-wise uniform convergence. Therefore, the natural metric for evaluating convergence is the \(L^\infty\)-error. We approximate this error by a point-wise maximum between \(u_h\) and \(u_\epsilon\) through
\(
	\|u_h-u_\epsilon\|_{L^\infty} \approx \max_{\bm x\in[0,1]_h^2} |u_h(\bm x) - u_\epsilon(\bm x)|.
\)
Since we only consider \(h\) as integer multiples of \(\epsilon\), we have hat \([0,1]_h^2\subset[0,1]_\epsilon^2\), meaning that this can be evaluated exactly. 

The convergence plots can be seen in \cref{fig:convergence:value}. We seem to have numerical convergence for all variables. Among the semi-discretised schemes, \eqref{eq:vinf} performs the best for all test problems. Apart from \cref{exp:c}, which is to some extent less interesting anyways, the scheme \eqref{eq:vinf} also performs better than the discretised dynamic programming. 

In \cite{calder2017}, it was demonstrated that the schemes based on \(D(u^2)\) have a higher numerical convergence rate than the schemes based on \(Du\). At first glance, we do not seem to have this property. However, the difference between the schemes becomes apparent in \cref{exp:b}, where the convergence rate of the scheme \eqref{eq:uinf} flattens out for \(h^{-1}\geq 10^{3}\). There are multiple factors contributing to the error of the schemes: the regularity of \(u\) (not being Lipschitz), the local variation of \(f\) and the number of shocks apparent in the value function. In \cite{calder2017}, the problems considered were very regular, with little to no variation in \(f\) and at most one shock solution. The \cref{exp:a,exp:b,exp:c} are substantially more complex, meaning that the error contributed from the lack of Lipschitz continuity of \(u\) is in most cases irrelevant. 
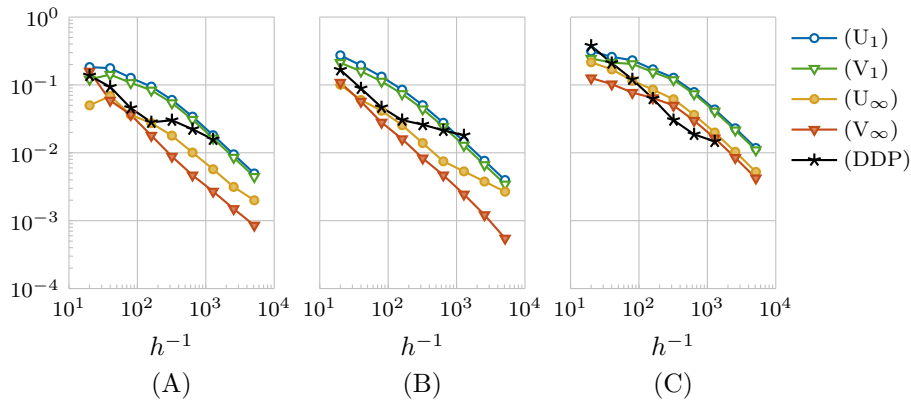
\begin{figure}[ht]
    \centering

\begin{tikzpicture}[
	trim axis group left,
	trim axis group right,
	line label/.style={%
		text height=4.79pt,
		text depth=0pt,
		anchor=west,
		inner sep=1pt,
		font=\scriptsize,
	}
]
	\begin{groupplot}[
		convergence plot,
		ymin = 1e-4,
		ymax = 1e0,
		group style={
			group size=3 by 1,
			horizontal sep=0.6cm,
			yticklabels at=edge left,
			xticklabels at=edge bottom,
			xlabels at=edge bottom,
			ylabels at=edge left,
		},
	]

	\nextgroupplot[
		title = \labelcref{exp:a},
	]
		\addplot table [y index = 2] {data/ex1/u1.txt};
		\addplot table [y index = 2] {data/ex1/v1.txt};
		\addplot table [y index = 2] {data/ex1/uinf.txt};
		\addplot table [y index = 2] {data/ex1/vinf.txt};
		\addplot table [y index = 2] {data/ex1/ddp.txt};
	\nextgroupplot[
		title = \labelcref{exp:b},
		y axis line style = {draw=none}, 
		y tick style = {draw=none},
	]
		\addplot table [y index = 2] {data/ex2/u1.txt};
		\addplot table [y index = 2] {data/ex2/v1.txt};
		\addplot table [y index = 2] {data/ex2/uinf.txt};
		\addplot table [y index = 2] {data/ex2/vinf.txt};
		\addplot table [y index = 2] {data/ex2/ddp.txt};
		
	\nextgroupplot[
		title = \labelcref{exp:c},
		y axis line style = {draw=none}, 
		y tick style = {draw=none},
	]
		\addplot table [y index = 2] {data/ex3/u1.txt};
		\addplot table [y index = 2] {data/ex3/v1.txt};
		\addplot table [y index = 2] {data/ex3/uinf.txt};
		\addplot table [y index = 2] {data/ex3/vinf.txt};
		\addplot table [y index = 2] {data/ex3/ddp.txt};
	
		\legend{
			\labelcref{eq:u1},
			\labelcref{eq:v1},
			\labelcref{eq:uinf},
			\labelcref{eq:vinf},
			\labelcref{eq:ddp},
		}
	\end{groupplot}
\end{tikzpicture}
    \caption{Convergence of \(u_h\) for \cref{exp:a,exp:b,exp:c}.}
    \label{fig:convergence:value}
\end{figure}

\subsection{Convergence of the Geodesic Distance}

By construction, we have that \(J(\bm\varphi) = u(\bm 1)\) whenever \(\bm\varphi\) is optimal. This gives us two ways to approximate the shape space distance:
\begin{align*}
    \distS([c_1],[c_2]) &\approx \arccos u_h(\bm1), \\
    \distS([c_1],[c_2]) &\approx \arccos J_h(\bm\varphi_h).
\end{align*}
For the fully discretised schemes, these quantities are the same by construction of the scheme. For the semi-discretised schemes, however, these are different quantities and might have different convergence properties.
The approximations were computed for each scheme and step size \(h\). For \cref{exp:a,exp:b}, we measured the error by comparison with \(\arccos u_h(\bm 1)\). For \cref{exp:c}, we have the exact solution \(\arccos u(\bm1) = \arccos J(\bm\varphi) = 0\). Convergence plots can be found in figures \ref{fig:convergence:value1} and \ref{fig:convergence:objective}, respectively.

We seem to have numerical convergence for all methods considered. We observe some cancellation effects, particularly for schemes \ref{eq:uinf} and \ref{eq:ddp} for \cref{exp:a} and \ref{eq:vinf} for \cref{exp:b}. The scheme \ref{eq:vinf} performs the best among the semi-discretised schemes while the fully discretised scheme has quite variable convergence properties.  Generally, it is hard to determine the exact convergence properties as we are essentially solving a PDE, but only measure convergence of the solution at a single point. For \cref{exp:c}, we only seem to have an \(\mathcal O(\sqrt h)\) convergence rate for the semi-discretised schemes. This is due to the non-differentiability of \(\arccos J\) at \(J=1\), which only occurs when the shape space distance is zero.\footnote{The function $\arccos J$ is also non-differentiable at $J = -1$. This value, however, can never occur as the solution of the optimisation problem.}

For \(J_h(\bm\varphi_h)\), all semi-discretised schemes perform almost identically. This might be due to the simple backtracking scheme we have proposed. Higher order backtracking schemes were tested without any significant improvement. For \cref{exp:a}, the convergence is too non-regular for a convergence rate to be estimated, for \cref{exp:b}, we seem to have a superlinear numerical convergence rate, and for \cref{exp:c}, we have a linear numerical convergence rate. Note also that apart from \ref{eq:vinf}, the distance estimates based on \(J_h(\bm\varphi_h)\) are more accurate than those based on \(u_h(\bm 1)\). Finally, also for the distance estimate based on \(u_h(\bm 1)\), we have worse convergence properties for \cref{exp:c} compared to \cref{exp:a,exp:b}. Again, this is explained by the non-differentiability of \(\arccos\). Consequently, we expect the schemes to perform worse for curves with equal shapes than for curves with non-zero shape space distance. 
\begin{figure}[ht]
    \centering

\begin{tikzpicture}[
	trim axis group left,
	trim axis group right,
	line label/.style={%
		text height=4.79pt,
		text depth=0pt,
		anchor=west,
		inner sep=1pt,
		font=\scriptsize,
	}
]
	\begin{groupplot}[
		convergence plot,
		ymin = 1e-4,
		ymax = 1e0,
		group style={
			group size=3 by 1,
			horizontal sep=0.6cm,
			yticklabels at=edge left,
			xticklabels at=edge bottom,
			xlabels at=edge bottom,
			ylabels at=edge left,
		},
	]

	\nextgroupplot[
		title = \labelcref{exp:a},
	]
		\addplot table [y index = 3] {data/ex1/u1.txt};
		\addplot table [y index = 3] {data/ex1/v1.txt};
		\addplot table [y index = 3] {data/ex1/uinf.txt};
		\addplot table [y index = 3] {data/ex1/vinf.txt};
		\addplot table [y index = 3] {data/ex1/ddp.txt};
	\nextgroupplot[
		title = \labelcref{exp:b},
		y axis line style = {draw=none}, 
		y tick style = {draw=none},
	]
		\addplot table [y index = 3] {data/ex2/u1.txt};
		\addplot table [y index = 3] {data/ex2/v1.txt};
		\addplot table [y index = 3] {data/ex2/uinf.txt};
		\addplot table [y index = 3] {data/ex2/vinf.txt};
		\addplot table [y index = 3] {data/ex2/ddp.txt};
		
	\nextgroupplot[
		title = \labelcref{exp:c},
		y axis line style = {draw=none}, 
		y tick style = {draw=none},
	]
		\addplot table [y index = 3] {data/ex3/u1.txt};
		\addplot table [y index = 3] {data/ex3/v1.txt};
		\addplot table [y index = 3] {data/ex3/uinf.txt};
		\addplot table [y index = 3] {data/ex3/vinf.txt};
		\addplot table [y index = 3] {data/ex3/ddp.txt};
	
		\legend{
			\labelcref{eq:u1},
			\labelcref{eq:v1},
			\labelcref{eq:uinf},
			\labelcref{eq:vinf},
			\labelcref{eq:ddp},
		}
	\end{groupplot}
\end{tikzpicture}
    \caption{Convergence of \(\arccos u_h(\bm 1)\) for \cref{exp:a,exp:b,exp:c}.}
    \label{fig:convergence:value1}
\end{figure}
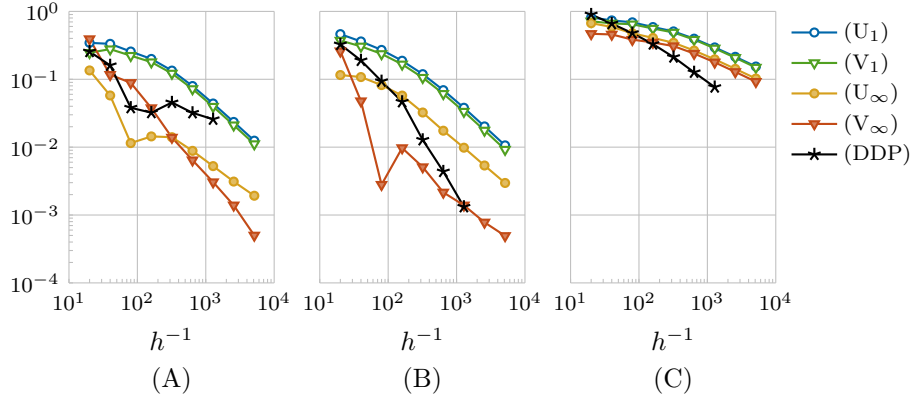
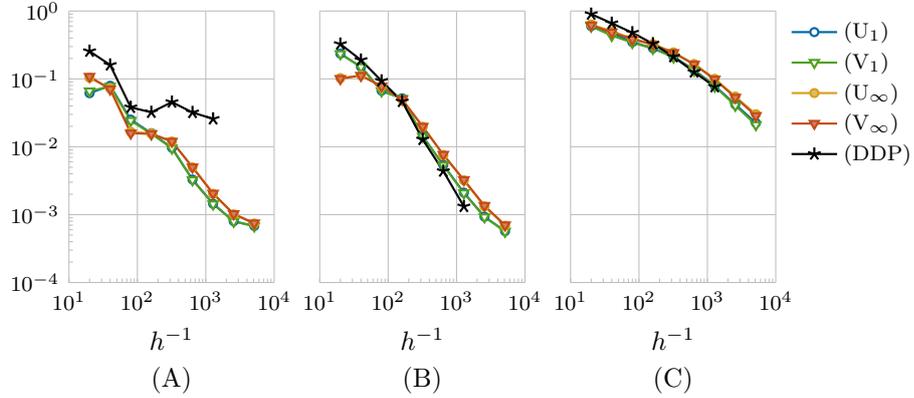
\begin{figure}[ht]
    \centering

\begin{tikzpicture}[
	trim axis group left,
	trim axis group right,
	line label/.style={%
		text height=4.79pt,
		text depth=0pt,
		anchor=west,
		inner sep=1pt,
		font=\scriptsize,
	}
]
	\begin{groupplot}[
		convergence plot,
		ymin = 1e-4,
		ymax = 1e0,
		group style={
			group size=3 by 1,
			horizontal sep=0.6cm,
			yticklabels at=edge left,
			xticklabels at=edge bottom,
			xlabels at=edge bottom,
			ylabels at=edge left,
		},
	]

	\nextgroupplot[
		title = \labelcref{exp:a},
	]
		\addplot table [y index = 4] {data/ex1/u1.txt};
		\addplot table [y index = 4] {data/ex1/v1.txt};
		\addplot table [y index = 4] {data/ex1/uinf.txt};
		\addplot table [y index = 4] {data/ex1/vinf.txt};
		\addplot table [y index = 4] {data/ex1/ddp.txt};
	\nextgroupplot[
		title = \labelcref{exp:b},
		y axis line style = {draw=none},
		y tick style = {draw=none},
	]
		\addplot table [y index = 4] {data/ex2/u1.txt};
		\addplot table [y index = 4] {data/ex2/v1.txt};
		\addplot table [y index = 4] {data/ex2/uinf.txt};
		\addplot table [y index = 4] {data/ex2/vinf.txt};
		\addplot table [y index = 4] {data/ex2/ddp.txt};
		
	\nextgroupplot[
		title = \labelcref{exp:c},
		y axis line style = {draw=none},
		y tick style = {draw=none},
	]
		\addplot table [y index = 4] {data/ex3/u1.txt};
		\addplot table [y index = 4] {data/ex3/v1.txt};
		\addplot table [y index = 4] {data/ex3/uinf.txt};
		\addplot table [y index = 4] {data/ex3/vinf.txt};
		\addplot table [y index = 4] {data/ex3/ddp.txt};
	
		\legend{
			\labelcref{eq:u1},
			\labelcref{eq:v1},
			\labelcref{eq:uinf},
			\labelcref{eq:vinf},
			\labelcref{eq:ddp},
		}
	\end{groupplot}
\end{tikzpicture}
    \caption{Convergence of \(\arccos J_h(\varphi_h)\) for \cref{exp:a,exp:b,exp:c}.}
    \label{fig:convergence:objective}
\end{figure}

\subsection{Convergence of the Geodesics}
Although we have numerical convergence of the geodesic distance estimate, this need not imply numerical convergence of the geodesics. Therefore, we consider numerical convergence of the geodesics as well. Consider the two approximate geodesics
\begin{align*}
	\gamma_h(\tau) &= w_h(1-\tau)q_{1,h} + w_h(\tau) q_{2,h}, \\
	\gamma_\epsilon(\tau) &=  w_\epsilon(1-\tau)q_{1,\epsilon} + w_\epsilon(\tau) q_{2,\epsilon}.
\end{align*}
To measure the difference between these geodesics, we use the maximal pre-shape distance over \(\tau\). Since the unit sphere distance is at most \(\pi\) times larger than the \(L^2\) distance, we have that
\begin{align*}
	\max_{\tau\in[0,1]} \arccos\langle\gamma_h(\tau), \gamma_\epsilon(\tau)\rangle_{L^2}
	&\leq \max_{\tau\in[0,1]} \pi\|\gamma_h(\tau) - \gamma_\epsilon(\tau)\|_{L^2} \\
		&= \pi\max\big\{
			\|\gamma_h(0) - \gamma_\epsilon(0)\|_{L^2},\,
			\|\gamma_h(1) - \gamma_\epsilon(1)\|_{L^2}
		\big\} \\
		&= \pi\max\big\{
			\|q_{1,h}-q_{1,\epsilon}\|_{L^2},\,
			\|q_{2,h}-q_{2,\epsilon}\|_{L^2}
		\big\}.
\end{align*}
In other words, we can easily compute an upper bound to the maximal unit sphere distance between the geodesics. Note that it would be even better to use the maximal \emph{shape space} distance between the geodesics. However, since the shape space distance requires the minimisation of the pre-shape distance, the upper bound is also an upper bound for the shape space distance.

From the convergence plots in \cref{fig:convergence:geodesics}, we observe numerical convergence. Again, we have no observable difference between the semi-discretised schemes. However, in this case, the semi-discretised schemes perform better than the discretised dynamic programming. 
\begin{figure}[ht]
    \centering

\begin{tikzpicture}[
	trim axis group left,
	trim axis group right,
	line label/.style={%
		text height=4.79pt,
		text depth=0pt,
		anchor=west,
		inner sep=1pt,
		font=\scriptsize,
	}
]
	\begin{groupplot}[
		convergence plot,
		ymin = 1e-3,
		ymax = 1e1,
		group style={
			group size=3 by 1,
			horizontal sep=0.6cm,
			yticklabels at=edge left,
			xticklabels at=edge bottom,
			xlabels at=edge bottom,
			ylabels at=edge left,
		},
	]

	\nextgroupplot[
		title = \labelcref{exp:a},
	]
		\addplot table [y index = 5] {data/ex1/u1.txt};
		\addplot table [y index = 5] {data/ex1/v1.txt};
		\addplot table [y index = 5] {data/ex1/uinf.txt};
		\addplot table [y index = 5] {data/ex1/vinf.txt};
		\addplot table [y index = 5] {data/ex1/ddp.txt};
	\nextgroupplot[
		title = \labelcref{exp:b},
		y axis line style = {draw=none}, 
		y tick style = {draw=none},
	]
		\addplot table [y index = 5] {data/ex2/u1.txt};
		\addplot table [y index = 5] {data/ex2/v1.txt};
		\addplot table [y index = 5] {data/ex2/uinf.txt};
		\addplot table [y index = 5] {data/ex2/vinf.txt};
		\addplot table [y index = 5] {data/ex2/ddp.txt};
		
	\nextgroupplot[
		title = \labelcref{exp:c},
		y axis line style = {draw=none}, 
		y tick style = {draw=none},
	]
		\addplot table [y index = 5] {data/ex3/u1.txt};
		\addplot table [y index = 5] {data/ex3/v1.txt};
		\addplot table [y index = 5] {data/ex3/uinf.txt};
		\addplot table [y index = 5] {data/ex3/vinf.txt};
		\addplot table [y index = 5] {data/ex3/ddp.txt};
	
		\legend{
			\labelcref{eq:u1},
			\labelcref{eq:v1},
			\labelcref{eq:uinf},
			\labelcref{eq:vinf},
			\labelcref{eq:ddp},
		}
	\end{groupplot}
\end{tikzpicture}
    \caption{Convergence of \(\gamma_h\) for \cref{exp:a,exp:b,exp:c}.}
    \label{fig:convergence:geodesics}
\end{figure}
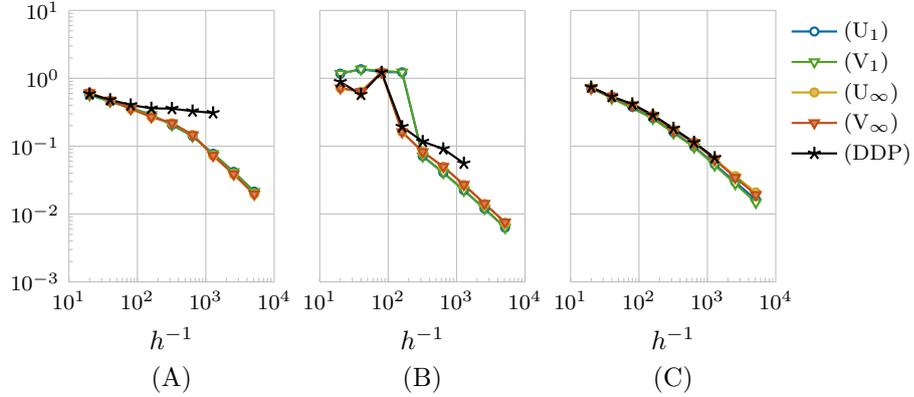

\subsection{Computational Complexity}
\label{sec:exp:complexity}
Until now, we have evaluated performance in terms of error vs step size. However, there is a significant difference in computational complexity between the semi-discretised and the fully discretised methods. For the fully discretised dynamic programming scheme, the computational complexity is \(\mathcal O(|A_h|N^{2})\), while for the semi-discretised schemes, the computational complexity is \(\mathcal O(N^{2})\). 
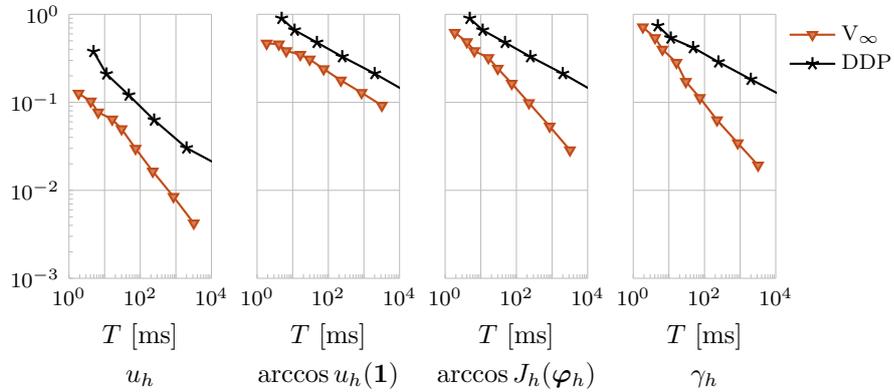
\begin{figure}[ht]
    \centering

\begin{tikzpicture}[
	trim axis group left,
	trim axis group right,
	line label/.style={%
		text height=4.79pt,
		text depth=0pt,
		anchor=west,
		inner sep=1pt,
		font=\scriptsize,
	}
]
	\begin{groupplot}[
		convergence plot,
		xmin=1,
		xlabel={\(T\) [ms]},
		ymin = 1e-3,
		ymax = 1e0,
		width=1.875cm,
		height=3.5  cm,
		max space between ticks=10,
		xticklabels={{\(10^0\)},,{\(10^2\)},,{\(10^4\)}},
		group style={
			group size=4 by 1,
			horizontal sep=0.6cm,
			yticklabels at=edge left,
			xticklabels at=edge bottom,
			xlabels at=edge bottom,
			ylabels at=edge left,
		},
	]
	
	\nextgroupplot[
		title = {\(u_h\vphantom(\)},
	]
		\addplot [line 4] table [x index=1,y index = 2] {data/ex3/vinf.txt};
		\addplot [line 5] table [x index=1,y index = 2] {data/ex3/ddp.txt};
	\nextgroupplot[
		title = {\(\arccos u_h(\bm 1)\)},
		y axis line style = {draw=none}, 
		y tick style = {draw=none},
	]
		\addplot [line 4] table [x index=1,y index = 3] {data/ex3/vinf.txt};
		\addplot [line 5] table [x index=1,y index = 3] {data/ex3/ddp.txt};
	\nextgroupplot[
		title = {\(\arccos J_h(\bm\varphi_h)\)},
		y axis line style = {draw=none}, 
		y tick style = {draw=none},
	]
		\addplot [line 4] table [x index=1,y index = 4] {data/ex3/vinf.txt};
		\addplot [line 5] table [x index=1,y index = 4] {data/ex3/ddp.txt};
	\nextgroupplot[
		title = {\(\gamma_h\vphantom(\)},
		y axis line style = {draw=none}, 
		y tick style = {draw=none},
	]
		\addplot [line 4] table [x index=1,y index = 5] {data/ex3/vinf.txt};
		\addplot [line 5] table [x index=1,y index = 5] {data/ex3/ddp.txt};
			
	\legend{
		\ref{eq:vinf},
		\ref{eq:ddp},
	}
	
	\end{groupplot}
\end{tikzpicture}
    \caption{Work-precision diagrams for \cref{exp:c}. \(T\) denotes the computation time.}
    \label{fig:convergence:complexity}
\end{figure}

To implement the schemes, we used Python using NumPy with vectorised updates. Even though the implementation is to some extent na\"ive, it is useful to evaluate the accuracy of the schemes vs the computation times. From the three test problems, it is \cref{exp:c} where the fully discretised scheme performs best compared to the semi-discretised schemes. Work-precision diagrams for this problem are visualised in \cref{fig:convergence:complexity}. As one can see, the semi-discretised scheme \eqref{eq:vinf} performs significantly better than the fully discretised method.


\section{Conclusion}
In this article, we have shown how PDE based method can be applied to the computation of shape space distances of open shapes. The method has global convergence and runs in \(\mathcal O(N^2)\) time, which is strictly better than existing global solvers. Additionally, the numerical experiments indicate a linear convergence in practice, although we expect a lower theoretical convergence rate.

First, we presented a family of schemes which generalises the schemes of \cite{calder2013,calder2017}. These are based on the Hamilton-Jacobi-Bellman equation for the value function of the problem. However, whereas the schemes of \cite{calder2013,calder2017} approximate the gradient of the value function using finite difference approximations, we approximate its directional derivatives. This allows for greater flexibility in the construction of the schemes. The resulting family of schemes has theoretical convergence, and we show that two instances of the scheme are more accurate than previous approaches. 

In conjunction with the schemes for the value function, we presented a backtracking scheme to obtain the solution of the reparametrisation problem. This is then used to estimate the shape space geodesics numerically. For different problems, the scheme seems to converge numerically, and the work-precision efficiency is better than that of previous global solvers.

From here, there is a number of interesting topics for future work, including the following:
\begin{itemize}
    \item Assessment of the typical \(\mathcal O(\sqrt h)\) convergence rate for the HJB schemes, similar to \cite[schemes S2, S3]{calder2017}.
	\item Assessing theoretical convergence of the backtracking method.
	\item Constructing schemes such as \eqref{eq:scheme} for general HJB equations.
	\item Construction of iterative solvers with adaptive grid refinement, where the HJB equation is solved on smaller and smaller strips around the solution of the reparametrisation problem, as has been done with great success for the fully discretised schemes \cite{dogan2015,dogan2016,dogan2021}.
\end{itemize}

\bibliographystyle{plain}
\bibliography{bib}


\appendix
\section{Approximating the SRVTs}
\label{sec:appf}
The schemes presented in this article are based on exact computation of the forcing term \(f(\bm x)\), which in turns requires access to the SRVTs \(q_1\) and \(q_2\). Whenever these are not available, we can use finite difference approximations of the curves \(c_1\) and \(c_2\). Here, we suggest using backward differences of the form
\[
	q_i(t) \approx
	\frac{c_i(t)-c_i(t-h)}{\sqrt{h|c_i(t)-c_i(t-h)|}},
\]
leading to the approximation
\[
	hf(\bm x) \approx
	\max\Biggl\{\biggl\langle
	\frac{c_1(x_1)-c_1(x_1-h)}{\sqrt{|c_1(x_1)-c_1(x_1-h)|}}
	,
	\frac{c_2(x_2)-c_2(x_2-h)}{\sqrt{|c_2(x_2)-c_2(x_2-h)|}}
	\biggr\rangle
	,\
	0
	\Biggr\}.
\]
For the fully discretised schemes, we suggest using backwards differences of the form
\[
	q_i(t)\sqrt{k} \approx
	\frac{c_i(t)-c_i(t-kh)}{\sqrt{h|c_i(t)-c_i(t-kh)|}},
\]
leading to the approximation
\[
	hf(\bm x)\sqrt{kl} \approx
	\max\Biggl\{\biggl\langle
	\frac{c_1(x_1)-c_1(x_1-kh)}{\sqrt{|c_1(x_1)-c_1(x_1-kh)|}}
	,
	\frac{c_2(x_2)-c_2(x_2-lh)}{\sqrt{|c_2(x_2)-c_2(x_2-lh)|}}
	\biggr\rangle
	,\
	0
	\Biggr\}.
\]
As long as the curves are immersions, i.e., that \(\lvert c_i'\rvert>0\) everywhere, these are consistent approximations, meaning that the proofs for convergence still hold. Moreover, we find that these approximations actually give better convergence properties for all implementations of the schemes.

\section{Convergence for Fully Discretised Schemes}
\label{sec:app}

We express the scheme \eqref{eq:scheme3sol} as the solution of \(S_h=0\) with
\begin{equation}
	\label{eq:scheme3}
	S_h = \max_{\bm\alpha\in A_h} \frac{u(\bm x-h\bm\alpha)-u(\bm x)}{h|\bm\alpha|} + f(\bm x)\frac{\sqrt{\alpha_1\alpha_2}}{|\bm\alpha|}.
\end{equation}
\begin{assumption}
	\label{asm:ah}
	\(A_h\) satisfies the following:
	\begin{subassumption}
		\item \(A_h \subset \mathbb N_0^2 \setminus \{\bm0\}\).
		\item \(A_h\) is finite for all \(h>0\). 
		\item \(\lim_{h\to0} \max_{\bm\alpha\in A_h} h|\bm\alpha| = 0\).
		\label{asm:bounded}
		\item For every \(\bm\beta\in A_2\) and \(\epsilon > 0\), there exists \(h_0>0\) such that for every \(0<h\leq h_0\), there is \(\bm\alpha \in A_h\) with \(|\bm\alpha/|\bm\alpha| - \bm\beta| <\epsilon\).
		\label{asm:consistent}
	\end{subassumption}
\end{assumption}

\begin{theorem}
	\label{thm:scheme3}
	Under \cref{asm:ah}, the scheme \eqref{eq:scheme3} is convergent. 
\end{theorem}
\begin{proof}
	The scheme satisfies the following properties:
	\begin{itemize}
		\item\emph{Monotonicity}: \(S_h\) is clearly non-decreasing in \(u(\bm y)\). 
		
		\item\emph{Stability}: 
		For all grid points \(\bm x\), there exists a grid point \(\bm y< \bm x\), such that 
		\begin{align*}
			u_h(\bm y) \leq u_h(\bm x) 
				&= u_h(\bm y) + f(\bm x)\sqrt{(x_1-y_1)(x_2-y_2)} \\
				&\leq u_h(\bm y) + \|f\|_\infty \frac12(x_1-y_1+x_2-y_2)
		\end{align*}
		using the Cauchy-Schwarz inequality. 
		Inductively, this gives that \(0\leq u(\bm x) \leq \|f\|_\infty\). 	
		
		\item\emph{Consistency}: We have that
		\[
		\begin{aligned}
			S_h\big(\bm y,\psi(\bm y)+\xi,\psi+\xi\big)
				&= 
					\max_{\bm\alpha\in A_h} 
					\frac{\psi(\bm y-h\bm\alpha) - \psi(\bm y)}{h|\bm\alpha|} + 
					f(\bm y)\frac{\sqrt{\alpha_1\alpha_2}}{|\bm\alpha|} \\
				&= 
					\max_{\bm\alpha\in A_h} 
					-D\psi(\bm y)\frac{\bm\alpha}{|\bm\alpha|} + O(h|\bm\alpha|)+ 
					f(\bm y)\frac{\sqrt{\alpha_1\alpha_2}}{|\bm\alpha|}. \\
		\end{aligned}
		\]
		Hence, due to  \cref{asm:bounded}, we have that \(\max_{\bm\alpha\in A_h}O(h|\bm\alpha|) = o(1)\). Moreover, we have that \(D\psi\) and \(f\) is uniformly continuous in \(y\). This, combined with \cref{asm:consistent}, gives that
		\[
			\lim_{\subalign{h&\to0\\\bm y&\to\bm x\\\xi&\to0} }
			S_h\big(\bm y,\psi(\bm y)+\xi,\psi+\xi\big)= 
					\max_{\bm\alpha\in A_2} 
					-D\psi(\bm x)\bm\alpha+ 
					f(\bm x)\sqrt{\alpha_1\alpha_2}
		\]
	\end{itemize}
	Due to \cite[Theorem 2.1]{barles1991}, this proves convergence. 
\end{proof}

\end{document}